\documentclass[12pt,a4paper,reqno]{amsart}
\usepackage[utf8]{inputenc}
\usepackage{amsmath}
\usepackage{geometry}
\usepackage{amssymb, latexsym}
\usepackage{verbatim}
\usepackage{enumerate}
\usepackage{comment}
\usepackage{txfonts}

\usepackage{mathtools}
\usepackage[tableposition=top]{caption}
\usepackage{booktabs,dcolumn}

\newtheorem{theorem}{Theorem}[section]
\newtheorem{proposition}[theorem]{Proposition}

\newtheorem{lemma}[theorem]{Lemma}
\newtheorem{corollary}[theorem]{Corollary}
\newtheorem{conjecture}[theorem]{Conjecture}

\newtheorem{definition}[theorem]{Definition}
\newtheorem{remark}[theorem]{Remark}

\newtheorem{example}[theorem]{Example}

\newcommand\E{\mathbb{E}}
\newcommand\R{\mathbb{R}}
\newcommand\Z{\mathbb{Z}}
\newcommand\N{\mathbb{N}}
\newcommand\D{\mathbb{D}}
\newcommand\Disk{\mathbb{D}}
\newcommand\C{\mathbb{C}}

\newcommand\eps{\varepsilon}
\newcommand\plim{\mathop{\widetilde \lim}}

\renewcommand\P{\mathbf{P}}

\begin{document}
\title[Correlations of multiplicative functions]{The structure of correlations of multiplicative functions at almost all scales, with applications to the Chowla and Elliott conjectures}

\author{Terence Tao}
\address{Department of Mathematics, UCLA\\
405 Hilgard Ave\\
Los Angeles CA 90095\\
USA}
\email{tao@math.ucla.edu}

\author{Joni Ter\"av\"ainen}
\address{Mathematical Institute, University of Oxford\\ Radcliffe Observatory Quarter, Woodstock Rd\\
Oxford OX2 6GG\\
UK}
\email{joni.teravainen@maths.ox.ac.uk}

\begin{abstract}	We study the asymptotic behaviour of higher order correlations
$$ \E_{n \leq X/d} g_1(n+ah_1) \cdots g_k(n+ah_k)$$
as a function of the parameters $a$ and $d$, where $g_1,\dots,g_k$ are bounded multiplicative functions, $h_1,\dots,h_k$ are integer shifts, and $X$ is large.  Our main structural result asserts, roughly speaking, that such correlations asymptotically vanish for almost all $X$ if $g_1 \cdots g_k$ does not (weakly) pretend to be a twisted Dirichlet character $n \mapsto \chi(n)n^{it}$, and behave asymptotically like a multiple of $d^{-it} \chi(a)$ otherwise. This extends our earlier work on the structure of logarithmically averaged correlations, in which the $d$ parameter is averaged out and one can set $t=0$. Among other things, the result enables us to establish special cases of the Chowla and Elliott conjectures for (unweighted) averages at almost all scales; for instance, we establish the $k$-point Chowla conjecture $ \E_{n \leq X} \lambda(n+h_1) \cdots \lambda(n+h_k)=o(1)$ for $k$ odd or equal to $2$ for all scales $X$ outside of a set of zero logarithmic density.
\end{abstract}

\maketitle

\section{Introduction}

\subsection{The Chowla and Elliott conjectures}

Define a \emph{$1$-bounded multiplicative function} to be a function $g: \N \to \D$ from the natural numbers $\N \coloneqq \{1,2,\dots\}$ to the unit disk $\Disk \coloneqq \{z \in \C: |z| \leq 1 \}$ satisfying $g(nm) = g(n) g(m)$ whenever $n,m$ are coprime.  If in addition $g(nm)=g(n)g(m)$ for all $n,m \in \N$, we say that $g$ is \emph{completely multiplicative}. In addition, we adopt the convention that $g(n)=0$ when $n$ is zero or a negative integer.

This paper is concerned with the structure of higher order correlations of such functions.  To describe our results, we need some notation for a number of averages.

\begin{definition}[Averaging notation]  Let $f: A \to \C$ be a function defined on a non-empty finite set $A$.
\begin{itemize}
\item[(i)]  (Unweighted averages)  We define
$$ \E_{n \in A} f(n) \coloneqq \frac{\sum_{n \in A} f(n)}{\sum_{n \in A} 1}.$$
\item[(ii)]  (Logarithmic averages)  If $A$ is a subset of the natural numbers $\N$, we define
$$ \E_{n \in A}^{\log} f(n) \coloneqq \frac{\sum_{n \in A} \frac{f(n)}{n}}{\sum_{n \in A} \frac{1}{n}}.$$
\item[(iii)]  (Doubly logarithmic averages)  If $A$ is a subset of the natural numbers $\N$, we define
$$ \E_{n \in A}^{\log\log} f(n) \coloneqq \frac{\sum_{n \in A} \frac{f(n)}{n \log(1+n)}}{\sum_{n \in A} \frac{1}{n \log(1+n)}}.$$
\end{itemize}
Of course, the symbol $n$ can be replaced here by any other free variable.  For any real number $X \geq 1$, we use $\E_{n \leq X} f(n)$ as a synonym for $\E_{n \in \N \cap [1,X]} f(n)$, and similarly for $\E_{n \leq X}^{\log} f(n)$ and $\E_{n \leq X}^{\log\log} f(n)$.  If we use the symbol $p$ (or $p_1, p_2$, etc.) instead of $n$, we implicitly restrict $p$ to the set of primes ${\P} \coloneqq \{2,3,5,7,\dots\}$, thus for instance for $X \geq 2$, $\E_{p \leq X} f(p)$ is a synonym for $\E_{p \in {\P} \cap [2,X]} f(p)$, and similarly for $\E_{p \leq X}^{\log} f(p)$.  
\end{definition}

\begin{remark}  The use of $\log(1+n)$ in the $\E^{\log\log}$ notation instead of $\log n$ is only in order to avoid irrelevant divergences at $n=1$, and the shift by $1$ may otherwise be ignored.  Because of the prime number theorem, prime averages such as $\E_{p \leq X} f(p)$ are often of ``comparable strength'' to logarithmic averages $\E_{n \leq X}^{\log} f(n)$, and similarly logarithmic prime averages such as $\E_{p \leq X}^{\log} f(p)$ are of comparable strength to $\E_{n \leq X}^{\log\log} f(n)$.  See Lemma \ref{chug} for a more precise statement.  
\end{remark}

Following Granville and Soundararajan \cite{gs}, given two $1$-bounded multiplicative functions $f,g: \N \to \Disk$, and $X \geq 1$, we define the \emph{pretentious distance} $\D(f,g;X)$ between $f$ and $g$ up to scale $X$ by the formula
$$ \D( f, g; X) \coloneqq \left( \sum_{p \leq X} \frac{1 - \mathrm{Re}(f(p) \overline{g(p)})}{p} \right)^{1/2}.$$

It is conjectured that multiple correlations of $1$-bounded multiplicative functions should asymptotically vanish unless all of the functions involved ``pretend'' to be twisted Dirichlet characters in the sense of the pretentious distance.  More precisely, the following conjecture is essentially due to Elliott.

\begin{conjecture}[Elliott conjecture]\label{elliott-conj}  Let $g_1,\dots,g_k: \N \to \Disk$ be $1$-bounded multiplicative functions for some $k \geq 1$.  Assume that there exists $j\in \{1,\dots,k\}$ such that for every Dirichlet character $\chi$ one has
\begin{equation}\label{suptx}
 \inf_{|t| \leq X} \D(g_j, n \mapsto \chi(n)n^{it}; X) \to \infty
\end{equation}
as $X \to \infty$.
\begin{itemize}
\item[(i)]  (Unweighted Elliott conjecture)  If $h_1,\dots,h_k \in \Z$ are distinct integers, then
$$ \lim_{X\to \infty}\E_{n \leq X} g_1(n+h_1) \cdots g_k(n+h_k) = 0.$$
\item[(ii)]  (Logarithmically averaged Elliott conjecture)  If $h_1,\dots,h_k \in \Z$ are distinct integers, then
$$ \lim_{X\to \infty}\E_{n \leq X}^{\log} g_1(n+h_1) \cdots g_k(n+h_k) = 0.$$
\end{itemize}
\end{conjecture}

Conjecture \ref{elliott-conj}(i) was first stated by Elliott in \cite{elliott},
\cite{elliott2}, with the condition \eqref{suptx} weakened to the assertion that $\D(g_j, n \mapsto \chi(n)n^{it}; X) \to \infty$ for each fixed $t$, with no uniformity in $t$ assumed.  However, it was shown in \cite{mrt} that this version of the conjecture fails for a technical reason.  By summation by parts, Conjecture \ref{elliott-conj}(i) implies Conjecture \ref{elliott-conj}(ii).  At present, both forms of the Elliott conjecture are known for $k=1$ (thanks to Halász's theorem \cite{halasz}), while the $k=2$ case of the logarithmic Elliott conjecture was established in \cite{tao}.  Specialising the above conjecture to the case of the Liouville function\footnote{For the definitions of the standard multiplicative functions used in this paper, see Subsection \ref{subsec: not}.} $\lambda$, we recover the following conjecture of Chowla \cite{chowla}, together with its logarithmically averaged form.

\begin{conjecture}[Chowla conjecture]\label{chowla-conj}  Let $k \geq 1$ be a natural number.
\begin{itemize}
\item[(i)]  (Unweighted Chowla conjecture)  If $h_1,\dots,h_k \in \Z$ are distinct integers, then
$$ \lim_{X\to \infty}\E_{n \leq X} \lambda(n+h_1) \cdots \lambda(n+h_k) = 0.$$
\item[(ii)]  (Logarithmically averaged Chowla conjecture)  If $h_1,\dots,h_k \in \Z$ are distinct integers, then
$$ \lim_{X\to \infty}\E_{n \leq X}^{\log} \lambda(n+h_1) \cdots \lambda(n+h_k) = 0.$$
\end{itemize}
\end{conjecture}

Note that for $k=1$, the unweighted Chowla conjecture is equivalent to the prime number theorem, while the logarithmically averaged 1-point Chowla conjecture has a short elementary proof.  No further cases of the unweighted Chowla conjecture are currently known, but the logarithmically averaged Chowla conjecture has been established for $k=2$ in \cite{tao} and for all odd values of $k$ in \cite{tt} (with a second proof given in \cite{tt-odd}).
The logarithmically averaged Chowla conjecture is also known to be equivalent to the logarithmically averaged form of a conjecture of Sarnak \cite{sarnak}; see \cite{tao-higher}.  See also \cite{mrt} for a version of Elliott's conjecture where one averages over the shifts $h_i$.  One can also formulate an analogous version of Chowla's conjecture for the M\"obius function, for which very similar\footnote{If one generalises the Chowla conjecture by using affine forms $a_in+h_i$ instead of shifts $n+h_i$, then a simple sieving argument can be used to show the equivalence of such generalised Chowla conjectures for the Liouville function and their counterparts for the M\"obius function; we leave the details to the interested reader.} results are known.

In \cite{tt}, we obtained the following special case of the logarithmically averaged Elliott conjecture (Conjecture \ref{chowla-conj} (ii)).  We say that a $1$-bounded multiplicative function $f: \N \to \Disk$ \emph{weakly pretends} to be another $1$-bounded multiplicative function $g: \N \to \Disk$ if
$$ \lim_{X \to \infty} \frac{1}{\log\log X} \D(f,g;X)^2 = 0$$
or equivalently
$$ \sum_{p\leq X} \frac{1-\mathrm{Re}(f(p)\overline{g(p)})}{p} = o(\log \log X).$$

\begin{theorem}[Special case of logarithmically averaged Elliott]\label{lae} \cite[Corollary 1.6]{tt}  Let $k \geq 1$, and let $g_1,\dots,g_k: \N \to \Disk$ be $1$-bounded multiplicative functions such that the product $g_1 \cdots g_k$ does not weakly pretend to be any Dirichlet character $n \mapsto \chi(n)$.  Then for any integers $h_1,\dots,h_k$, one has
$$ \lim_{X\to \infty}\E_{n \leq X}^{\log} g_1(n+h_1) \cdots g_k(n+h_k) = 0.$$
\end{theorem}

In particular this establishes the logarithmically averaged Chowla conjecture for odd values of $k$.  This result was also recently used by Frantzikinakis and Host \cite{fh-furstenberg} to control the Furstenberg measure-preserving systems associated to $1$-bounded multiplicative functions, and to establish a version of the logarithmic Sarnak conjecture where the M\"obius function $\mu(n)$ is replaced by a $1$-bounded multiplicative function $g(n)$ and the topological dynamical system involved is assumed to be uniquely ergodic.

Theorem \ref{lae} was deduced from a more general structural statement about the correlation sequence $a \mapsto \lim_{X\to \infty}\E_{n \leq X}^{\log} g_1(n+ ah_1) \cdots g_k(n+ah_k)$ for $1$-bounded multiplicative functions $g_1,\dots,g_k$, where one now permits the product $g_1\cdots g_k$ to weakly pretend to be a Dirichlet character.  Here one runs into the technical difficulty that the asymptotic limits $\lim_{X\to \infty}\E_{n \leq X}^{\log}$ are not known \emph{a priori} to exist.  To get around this difficulty, the device of \emph{generalised limit functionals}\footnote{Alternatively, one could employ ultrafilter limits, or pass to subsequences in which all limits of interest exist.  The latter approach is for instance the one adopted in \cite{frantz-2}, \cite{fh-sarnak}, \cite{fh-furstenberg}.} was employed.  By a generalised limit functional we mean a bounded linear functional $\lim^*_{X \to \infty}: \ell^\infty(\N) \to \C$ which agrees with the ordinary limit functional $\lim_{X \to \infty}$ on convergent sequences, maps non-negative sequences to non-negative numbers, and which obeys the bound
$$ |\lim^*_{X \to \infty} f(X)| \leq \limsup_{X \to \infty} |f(n)|$$
for all bounded sequences $f$.  As is well known, the existence of such generalised limits follows from the Hahn-Banach theorem.  With these notations, we proved in \cite{tt} the following.

\begin{theorem}[Structure of logarithmically averaged correlation sequences]\label{Str}\cite[Theorem 1.1]{tt}
Let $k \geq 1$, and let $h_1,\dots,h_k$ be integers and $g_1,\dots, g_k:\N\to \Disk$ be $1$-bounded multiplicative functions.  Let $\lim^*_{X \to \infty}$ be a generalised limit functional.  Let $f: \Z \to \Disk$ denote the function
\begin{align}\label{eq18}
 f(a) \coloneqq \lim^{*}_{X\to \infty}\E_{n\leq X}^{\log} g_1(n+ah_1) \cdots g_k(n+ah_k).
\end{align}
\begin{itemize}
\item[(i)]  If the product $g_1 \cdots g_k$ does not weakly pretend to be a Dirichlet character, then $f$ is identically zero.
\item[(ii)]  If instead the product $g_1 \cdots g_k$ weakly pretends to be a Dirichlet character $\chi$, then $f$ is the uniform limit of periodic functions $F_i$, each of which is $\chi$-\textnormal{isotypic} in the sense that $F_i(ab) = F_i(a) \chi(b)$ whenever $a$ is an integer and $b$ is an integer coprime to the periods of $F_i$ and $\chi$.
\end{itemize}
\end{theorem}

Among other things, Theorem \ref{Str} yields Theorem \ref{lae} as a direct corollary.  Theorem \ref{lae} in turn can be used to establish various results about the distribution of consecutive values of $1$-bounded multiplicative functions; to give just one example, in \cite[Corollary 7.2]{tt} it was used to show that every sign pattern in $\{-1,+1\}^3$ occurred with logarithmic density $1/8$ amongst the Liouville sign patterns $(\lambda(n), \lambda(n+1), \lambda(n+2))$.

\subsection{From logarithmic averages to almost all ordinary averages}

It would be desirable if many of the above results for logarithmically averaged correlations such as $\E_{n \leq X}^{\log} g_1(n+h_1) \cdots g_k(n+h_k)$ could be extended to their unweighted counterparts such as $\E_{n \leq X} g_1(n+h_1) \cdots g_k(n+h_k)$.   However, such extensions cannot be automatic, since for instance the logarithmic averages $\E_{n \leq X}^{\log} n^{it}$ converge to $0$ for $t \neq 0$, but the unweighted averages $\E_{n \leq X} n^{it}$ diverge.  Similarly, the statement $\E_{n \leq X}^{\log} \lambda(n)=o(1)$ has a short and simple elementary proof\footnote{One can for example prove this by writing $\E_{n \leq X}^{\log} \lambda(n)=-\mathbb{E}_{p\leq y}\E_{n \leq X}^{\log} \lambda(n)p1_{p\mid n}+o_{y\to \infty}(1)$, and then using the Turán-Kubilius inequality to get rid of the $p1_{p\mid n}$ factor; we leave the details to the interested reader.}, whereas the unweighted analogue $\E_{n \leq X} \lambda(n)=o(1)$ is equivalent to the prime number theorem and its proofs are more involved. Moreover, one can show\footnote{More generally, one can use partial summation to show that, for any bounded real-valued sequence $a:\mathbb{N}\to \mathbb{R}$, if $\lim_{X\to \infty}\mathbb{E}_{n\leq X}^{\log}a(n)=\alpha$, then there exists an increasing sequence $X_i$ such that $\lim_{i\to \infty}\mathbb{E}_{n\leq X_i}a(n)=\alpha$. In particular, if the logarithmic Elliott conjecture holds, then the ordinary Elliott conjecture also holds in the case of real-valued functions along some subsequence of scales (which may depend on the functions involved).} that if, for example, the correlation limit $\lim_{X\to \infty} \mathbb{E}_{n\leq X}\lambda(n)\lambda(n+1)$ exists, then it has to be equal to $0$, which means that proving the mere existence of the limit captures the difficulty in the two-point unweighted Chowla conjecture.

Nevertheless, there are some partial results of this type in which control on logarithmic averages can be converted to control on unweighted averages for a subsequence of scales $X$.  For instance, in \cite{gkl} it is shown using ergodic theory techniques that if the logarithmically averaged Chowla conjecture holds for all $k$, then there exists an increasing sequence of scales $X_i$ such that the Chowla conjecture for all $k$ holds for $X$ restricted to these scales. This was refined in a blog post \cite{tao-blog} of the first author, where it was shown by an application of the second moment method that if the logarithmically averaged Chowla conjecture held for some even order $2k$, then the Chowla conjecture for order $k$ would hold for all scales $X$ outside of an exceptional set ${\mathcal X}\subset \mathbb{N}$ of logarithmic density zero, by which we mean that
$$ \lim_{X\to \infty}\E^{\log}_{n \leq X} 1_{\mathcal X}(n) = 0.$$
Unfortunately, as the only even number for which the logarithmically averaged Chowla conjecture is currently known to hold is $k=2$, this only recovers (for almost all scales) the $k=1$ case of the unweighted Chowla conjecture, which was already known from the prime number theorem.

At present, the restriction to logarithmic averaging in many of the above results is needed largely because it supplies (via the ``entropy decrement argument'') a certain approximate dilation invariance, which roughly speaking asserts the approximate identity
$$ g_1(p)\cdots g_k(p)\E_{n \leq X}^{\log} g_1(n+h_1) \cdots g_k(n+h_k) \approx \E_{n \leq X}^{\log} g_1(n+ph_1) \cdots g_k(n+ph_k)$$
for ``most'' primes $p$, and for extremely large values of $X$; see for instance \cite[Theorem 3.2]{fh-furstenberg} for a precise form of this statement, with a proof essentially provided in \cite[\S 3]{tt}.  However, an inspection of the entropy decrement argument reveals that it also provides an analogous identity for unweighted averages, namely that
\begin{equation}\label{nxp}
g_1(p)\cdots g_k(p) \E_{n \leq X} g_1(n+h_1) \cdots g_k(n+h_k) \approx \E_{n \leq X/p} g_1(n+ph_1) \cdots g_k(n+ph_k)
\end{equation}
for ``most'' primes $p$, and ``most'' extremely large values of $X$; see Proposition \ref{dollop} for a precise statement.  By using this form of the entropy decrement argument, we are able to obtain the following analogue of Theorem \ref{Str} for unweighted averages, which is the main technical result of our paper and is proven in Section \ref{main-sec}.

\begin{theorem}[Structure of unweighted correlation sequences]\label{main}  Let $k \geq 1$, and let $h_1,\dots,h_k$ be integers and $g_1,\dots, g_k:\N\to \Disk$ be $1$-bounded multiplicative functions.  Let $\lim^*_{X \to \infty}$ be a generalised limit functional.  
For each real number $d>0$, let $f_d: \Z \to \Disk$ denote the function
\begin{equation}\label{fda-def}
 f_d(a) \coloneqq \lim^*_{X \to \infty} \E_{n \leq X/d} g_1(n+ah_1) \cdots g_k(n+ah_k).
\end{equation}
\begin{itemize}
\item[(i)]  If the product $g_1 \cdots g_k$ does not weakly pretend to be any twisted Dirichlet character $n \mapsto \chi(n)n^{it}$, then 
$$\lim_{X\to \infty} \E^{\log\log}_{d \leq X} |f_d(a)| = 0$$
for all integers $a$.
\item[(ii)]  If instead the product $g_1 \cdots g_k$ weakly pretends to be a twisted Dirichlet character $n \mapsto \chi(n)n^{it}$, then there exists a function $f: \Z \to \Disk$ such that
\begin{equation}\label{hd}
 \lim_{X\to \infty}\E^{\log\log}_{d \leq X} |f_d(a) - f(a) d^{-it}| = 0
\end{equation}
for all integers $a$.  Furthermore, $f$ is the uniform limit of $\chi$-isotypic\footnote{That is, we have $F_i(ab)=F_i(a)\chi(b)$ for any integers $a$ and $b$ with $b$ coprime to the periods of $F_i$ and $\chi$.} periodic functions $F_i$.
\end{itemize}
\end{theorem}

We have defined $f_d$ for all real numbers $d>0$ for technical reasons, but we will primarily be interested in the behaviour of $f_d$ for natural numbers $d$; for instance, the averages $\lim_{X\to \infty}\E^{\log\log}_{d \leq X}$ appearing in the above theorem are restricted to this case.

Roughly speaking, the logarithmic correlation sequence $f(a)$ appearing in Theorem \ref{Str} is analogous to the average $\lim_{X\to \infty}\E^{\log\log}_{d \leq X} f_d(a)$ of the sequences appearing here (ignoring for this discussion the question of whether the limits exist).  These averages vanish when $t \neq 0$ in Theorem \ref{main}, and one basically recovers a form of Theorem \ref{Str}; but, as the simple example of averaging the single $1$-bounded multiplicative function $n \mapsto n^{it}$ already shows, in the $t \neq 0$ case it is possible for the $f_d(a)$ to be non-zero while the logarithmically averaged counterpart $f(a)$ vanishes.

By combining Theorem \ref{main} with a simple application of the Hardy--Littlewood maximal inequality, we can obtain several new cases of the unweighted Elliott and Chowla conjectures at almost all scales, as follows.

\begin{corollary}[Some cases of the unweighted Elliott conjecture at almost all scales]\label{elliott-1}  Let $k \geq 1$, and let $g_1,\dots, g_k:\N\to \Disk$ be $1$-bounded multiplicative functions.  Suppose that the product $g_1 \cdots g_k$ does not weakly pretend to be any twisted Dirichlet character $n \mapsto \chi(n) n^{it}$.  
\begin{itemize}
\item[(i)]  For any $h_1,\dots,h_k \in \Z$ and $\eps>0$, one has
$$ |\E_{n \leq X} g_1(n+h_1) \cdots g_k(n+h_k)| \leq \eps $$
for all natural numbers $X$ outside of a set $\mathcal{X}_{\varepsilon}$ of logarithmic Banach density zero, in the sense that
\begin{align}\label{eq6}
\lim_{\omega \to \infty} \sup_{X \geq\omega} \E^{\log}_{X/\omega \leq n \leq X} 1_{\mathcal{X}_{\varepsilon}}(n) = 0.
\end{align}
\item[(ii)]  There is a set ${\mathcal X}_0$ of logarithmic density zero, such that
$$ \lim_{X \to \infty; X \not \in {\mathcal X}_0} \E_{n \leq X} g_1(n+h_1) \cdots g_k(n+h_k)  = 0$$
for all $h_1,\dots,h_k \in \Z$.
\end{itemize}
\end{corollary}

\begin{remark}
We note that Corollary \ref{elliott-1} can be generalised to the case of dilated correlations
$$ \E_{n \leq X} g_1(q_1n+h_1) \cdots g_k(q_kn+h_k), $$
where $q_1,\ldots, q_k\in \mathbb{N}$. To see this, one applies exactly the same trick related to Dirichlet character expansions as in \cite[Appendix A]{tt}. Similarly, Corollary \ref{elliott-2} below generalises to the dilated case. We leave the details to the interested reader.
\end{remark}

\begin{remark}
We see by partial summation that if $f:\mathbb{N}\to \mathbb{C}$ is any bounded function such that for every $\varepsilon>0$ we have $\displaystyle |\lim_{X\to \infty; X\not \in \mathcal{X}_{\varepsilon}}\mathbb{E}_{n\leq X}f(n)|\leq \varepsilon$ for some set $\mathcal{X}_{\varepsilon}\subset \mathbb{N}$ of logarithmic Banach density $0$, then we also have the logarithmic correlation result $\displaystyle\limsup_{X\to \infty} |\mathbb{E}_{X/\omega(X)\leq n\leq X}^{\log}f(n)|\ll \varepsilon$ for any function $1\leq \omega(X)\leq X$ tending to infinity. Thus Corollary \ref{elliott-1} is a strengthening of our earlier result \cite[Corollary 1.6]{tt} on logarithmic correlation sequences. Similarly, Corollary \ref{elliott-2} below is a strengthening of \cite[Corollary 1.5]{tao}.
\end{remark}

\begin{remark}
The logarithmic density (or logarithmic Banach density) appearing in Corollaries \ref{elliott-1} and \ref{elliott-2} is the right density to consider in this problem. Namely, if one could show that the set $\mathcal{X}_0$ has \emph{asymptotic} density $0$, then $[1,\infty)\setminus \mathcal{X}_0$ would intersect every interval $[x,(1+\varepsilon)x]$ for all large $x$, which would easily imply (together with \eqref{eq13} below) that the unweighted correlation converges to zero without any exceptional scales.
\end{remark}

\begin{remark}  The twisted Dirichlet characters $\chi(n) n^{it}$ appear both in Conjecture \ref{elliott-conj} and in Theorems \ref{Str}, \ref{main}.  However, there is an interesting distinction as to how they appear; in Conjecture \ref{elliott-conj}, $t$ is allowed to be quite large (as large as $X$) and $\chi(n) n^{it}$ is associated to just a single multiplicative function $g_j$, while in Theorems \ref{Str}, \ref{main}, the quantity $t$ is independent of $X$ and is now associated to the product $g_1 \dots g_k$.  The dependence of $t$ on $X$ in Conjecture \ref{elliott-conj}(i) is necessary\footnote{In the case of the \emph{logarithmically averaged} Conjecture \ref{elliott-conj}(ii), in contrast, \eqref{suptx} might not be a necessary assumption, since the sequence of bad scales constructed in \cite[Theorem B.1]{mrt} is sparse and thus does not influence the behaviour of logarithmic averages. We thank the referee for this remark.}, as is shown in \cite{mrt}; roughly speaking, the individual $g_j$ can oscillate like $n^{it_j}$ for various large $t_j$ in such a fashion that these oscillations largely cancel and produce non-trivial correlations in the product $g_1(n+h_1) \dots g_k(n+h_k)$.  Meanwhile, Theorem \ref{main} asserts in some sense that the shifted product $g_1(n+h_1) \dots g_k(n+h_k)$ oscillates ``similarly to'' the unshifted product $g_1(n) \dots g_k(n)$, so in particular if the latter began oscillating like $n^{it}$ for increasingly large values of $t$ then the former product should exhibit substantial cancellation.
\end{remark}

The proof of Corollary \ref{elliott-1} is found in Section \ref{corollaries-sec}. So far, all of our results have concerned correlations where the product of the multiplicative functions involved is non-pretentious. In the case of two-point correlations, however, we can prove Corollary \ref{elliott-1} under the mere assumption that one of the multiplicative functions involved is non-pretentious, thus upgrading the logarithmic two-point Elliott conjecture in \cite{tao} to an unweighted version at almost all scales.

\begin{corollary}[The binary unweighted Elliott conjecture at almost all scales]\label{elliott-2}  Let $g_1, g_2:\N\to \Disk$ be $1$-bounded multiplicative functions, such that there exists $j\in \{1,2\}$ for which \eqref{suptx} holds as $X \to \infty$ for every Dirichlet character $\chi$.
\begin{itemize}
\item[(i)] For any distinct $h_1,h_2 \in \Z$ and $\eps > 0$, one has
$$ |\E_{n \leq X} g_1(n+h_1)  g_2(n+h_2)| \leq \eps $$
for all natural numbers $X$ outside of a set $\mathcal{X}_{\varepsilon}$ of logarithmic Banach density zero (in the sense of \eqref{eq6}).
\item[(ii)]  There is a set ${\mathcal X}_0$ of logarithmic density zero such that
$$ \lim_{X \to \infty; X \not \in {\mathcal X}_0} \E_{n \leq X} g_1(n+h_1)  g_2(n+h_2)  = 0$$
for all distinct $h_1,h_2 \in \Z$.
\end{itemize}
\end{corollary}

When specialised to the case of the Liouville function, the previous corollaries produce the following almost-all result.

\begin{corollary}[Some cases of the unweighted Chowla conjecture at almost all scales]\label{chow}  There is an exceptional set ${\mathcal X}_0$ of logarithmic density zero, such that
$$ \lim_{X \to \infty; X \not \in {\mathcal X}_0} \E_{n \leq X} \lambda(n+h_1) \cdots \lambda(n+h_k) = 0$$
for all natural numbers $k$ that are either odd or equal to $2$, and for any distinct integers $h_1,\dots,h_k$.  The same result holds if one replaces one or more of the copies of the Liouville function $\lambda$ with the M\"obius function $\mu$.
\end{corollary}

We establish these results in Section \ref{corollaries-sec}.
One can use these corollaries to extend some previous results involving the logarithmic density of sign patterns to now cover unweighted densities of sign patterns at almost all scales.  For instance, by inserting Corollary \ref{chow} into the proof of \cite[Corollary 1.10(i)]{tt}, one obtains the following.

\begin{corollary}[Liouville sign patterns of length three]  There is an exceptional set ${\mathcal X}_0$ of logarithmic density zero, such that
$$ \lim_{X \to \infty; X \not \in {\mathcal X}_0} \E_{n \leq X} 1_{(\lambda(n),\lambda(n+1),\lambda(n+2)) = (\epsilon_0, \epsilon_1, \epsilon_2)} = \frac{1}{8}$$
for all sign patterns $(\epsilon_0, \epsilon_1, \epsilon_2) \in \{-1,1\}^3$.
\end{corollary}

Similarly several other results in \cite{tt} and in \cite{tera-binary} can be generalised. For example, the result \cite[Theorem 1.16]{tera-binary} on the largest prime factors of consecutive integers can be upgraded to the following form.

\begin{corollary}[The largest prime factors of consecutive integers at almost all scales] \label{cor1} Let $P^{+}(n)$ be the largest prime factor of $n$ with $P^{+}(1):=1$. Then there is an exceptional set  $\mathcal{X}_0$ of logarithmic density $0$, such that
\begin{align}\label{eq10}
\lim_{X\to \infty; X\not \in \mathcal{X}_0} \mathbb{E}_{n\leq X}1_{P^{+}(n)<P^{+}(n+1)}=\frac{1}{2}.  \end{align}

\end{corollary}

The same equality with ordinary limit in place of the almost-all limit is an old conjecture formulated in the correspondence of Erd\H{o}s and Turán \cite[pp. 100--101]{sos}, \cite{erdos}. We remark on the proof of Corollary \ref{cor1} in Remark \ref{rem1}. In \cite[Theorem 1.6]{tera-binary} it was proved that  \eqref{eq10} holds for the logarithmic average $\mathbb{E}_{n\leq X}^{\log}$ (without any exceptional scales).

It would of course be desirable if we could upgrade ``almost all scales'' to ``all scales'' in the above results.  We do not know how to do so in general, however there is one exceptional (though conjecturally non-existent) case in which this is possible, namely if there are unusually few sign patterns in the multiplicative functions of interest.    We illustrate this principle with the following example.

\begin{theorem}[Few sign patterns implies binary Chowla conjecture]\label{sign-thm}  Suppose that for every $\eps>0$, there exist arbitrarily large natural numbers $K$ such that the set $\{ (\lambda(n+1), \dots, \lambda(n+K)): n \in \N \} \subset \{-1,+1\}^K$ of sign patterns of length $K$ has cardinality less than $\exp( \eps \frac{K}{\log K} )$.  Then, for any natural number $h$, one has
\begin{align*}
\lim_{X\to \infty}\E_{n \leq X} \lambda(n) \lambda(n+h) = 0.
\end{align*}
\end{theorem}

\begin{remark}
The best known lower bounds for the number $s(K)$ of sign patterns of length $K$ for the Liouville function are very far from $\exp(\varepsilon \frac{K}{\log K})$. It was shown by Matom\"aki, Radziwi\l{}\l{} and Tao \cite{mrt-sign} that $s(K)\geq K+5$, and Frantzikinakis and Host showed in \cite{fh-sarnak} that $s(K)/K\to \infty$ as $K\to \infty$, but the rate of growth is inexplicit in that result. This was very recently improved to $s(K)\gg K^2$ by McNamara \cite{mcnamara}. If one assumes the Chowla conjecture (in either the unweighted or logarithmically averaged forms), it is not difficult to conclude that in fact $s(K) =2^K$ for all $K$.
\end{remark}

We prove this result in Section \ref{sign-sec}.  Roughly speaking, the reason for this improvement is that the entropy decrement argument that is crucially used in the previous arguments becomes significantly stronger under the hypothesis of few sign patterns.  A similar result holds for the odd order cases of the Chowla conjecture if one assumes the sign pattern control for \emph{all} large $K$ (rather than for a sequence of arbitrarily large $K$) by adapting the arguments in  \cite{tt-odd}, but we do not do so here.  It is also possible to strengthen this theorem in a number of further ways (for instance, restricting attention to sign patterns that occur with positive upper density, or to extend to other $1$-bounded multiplicative functions than the Liouville function), but we again do not do so so here. 

One should view Theorem \ref{sign-thm} as stating that if there is ''too much structure'' in the Liouville sequence (in the sense that it has a small number of sign patterns), then the binary Chowla conjecture holds. This is somewhat reminiscent of various statements in analytic number theory that rely on the assumption of a Siegel zero; for example, Heath-Brown \cite{hb} proved that if there are Siegel zeros, then the twin prime conjecture (which is connected to the two-point Chowla conjecture) holds. Nevertheless, the proof of Theorem \ref{sign-thm} does not resemble that in \cite{hb}.

\subsection{Isotopy formulae}

The conclusion of Theorem \ref{main}(ii) asserts, roughly speaking, that $f_d(a)$ ``behaves like'' a multiple of $\chi(a) d^{-it}$ in a certain asymptotic sense.  The following corollary of that theorem makes this intuition a bit more precise. 

\begin{theorem}[Isotopy formulae] \label{isotopy}
Let $k\geq 1$, let $h_1,\ldots, h_k$ be integers and $g_1,\ldots, g_k:\mathbb{N}\to \mathbb{D}$ be $1$-bounded multiplicative functions. Suppose that the product $g_1\cdots g_k$ weakly pretends to be a twisted Dirichlet character $n\mapsto \chi(n)n^{it}$. 

\begin{itemize}
\item[(i)] (Archimedean isotopy) There exists an exceptional set $\mathcal{X}_0$ of logarithmic density zero, such that 
\begin{align*}
\lim_{X\to \infty; X\not \in \mathcal{X}_0}\left(\mathbb{E}_{n\leq X}g_1(n+h_1)\cdots g_k(n+h_k)-q^{it}\mathbb{E}_{n\leq X/q}g_1(n+h_1)\cdots g_k(n+h_k)\right)=0    
\end{align*}
for all rational numbers $q>0$.

\item[(ii)] (Non-archimedean isotopy) There exists an exceptional set $\mathcal{X}_0$ of logarithmic density zero, such that 
\begin{align*}
\lim_{X\to \infty; X\not \in \mathcal{X}_0}\left(\mathbb{E}_{n\leq X}g_1(n-ah_1)\cdots g_k(n-ah_k)-\chi(-1)\mathbb{E}_{n\leq X}g_1(n+ah_1)\cdots g_k(n+ah_k)\right)=0    
\end{align*}
for all integers $a$.
\end{itemize}
\end{theorem}

\begin{remark}
This generalises \cite[Theorem 1.2(iii)]{tt}, which implies $f(-a)=\chi(-1)f(a)$ with $f(a)$ a generalised limit of a logarithmic correlation defined in \eqref{eq18} (indeed, Theorem \ref{isotopy}(ii) implies by partial summation that $f(-a)=\chi(-1)f(a)$ in the notation of \eqref{eq18}). In \cite{tt}, we only considered logarithmically averaged correlations, and for such averages Theorem \ref{isotopy}(i) does not make sense, as logarithmic averages are automatically slowly varying. However, for unweighted averages Theorem \ref{isotopy}(i) gives nontrivial information about the behaviour of the correlation at nearby scales. 
\end{remark}

We give the proof of Theorem \ref{isotopy} in Section \ref{isotopy-sec}. We show in that section that, perhaps surprisingly, the non-archimedean isotopy formula (Theorem \ref{isotopy}(ii)) allows us to evaluate the correlations of some multiplicative functions whose product \emph{does pretend} to be a Dirichlet character. Among other things, we use the isotopy formula to prove a version of the even order logarithmic Chowla conjectures where we twist one of the copies of the Liouville function by a carefully chosen Dirichlet character and the shifts of $\lambda$ are consecutive.

\begin{corollary}[Even order correlations of a twisted Liouville function] \label{even-chowla} Let $k\geq 4$ be an even integer, and let $\chi$ be an odd Dirichlet character of period $k-1$ (there are $\frac{\varphi(k-1)}{2}$ such characters). Then there exists an exceptional set $\mathcal{X}_0$ of logarithmic density $0$, such that
\begin{align}\label{eq1}
\lim_{X\to \infty; X\not \in \mathcal{X}_0}\mathbb{E}_{n\leq X}\chi(n)\lambda(n)\lambda(n+a)\cdots \lambda(n+(k-1)a)=0   
\end{align}
for all integers $a$.
\end{corollary}

By partial summation, we see from \eqref{eq1} that we have the logarithmic correlation result 
\begin{align*}
\lim_{X \to \infty}\mathbb{E}_{n \leq X}^{\log} \chi(n)\lambda(n)\lambda(n+1)\cdots \lambda(n+k-1)=0,        
\end{align*}
which is already new. We stated Corollary \ref{even-chowla} only for even $k$, but of course the result also holds for odd $k$ by Corollary \ref{elliott-1}.

The assumption that $\chi$ is an odd character is crucial above, as will be seen in Section \ref{isotopy-sec}; the isotopy formulae are not able to say anything about the untwisted even order correlations of the Liouville function. 

We likewise show in Section \ref{isotopy-sec} that the archimedean isotopy formula (Theorem \ref{isotopy}(i)) gives a rather satisfactory description of the \emph{limit points} of the correlations
\begin{align}\label{eq19}
\mathbb{E}_{n\leq X}g_1(n+h_1)\cdots g_k(n+h_k),
\end{align}
where the product $g_1\cdots g_k$ weakly pretends to be a twisted Dirichlet character $n\mapsto \chi(n)n^{it}$ with $t\neq 0$. Indeed, our Theorem \ref{limits} shows that once one continuously excludes the scales at which the correlation \eqref{eq19} is close to zero, the argument of the quantity in \eqref{eq19} is in a sense uniformly distributed on the unit circle. This uniform distribution is indeed expected when $g_j$ are pretentious; for example, one has $\mathbb{E}_{n\leq X}n^{it}=\frac{X^{it}}{1+it}+o(1)$, which uniformly distributes on the circle of radius $1/|1+it|$ with respect to logarithmic density.

\subsection{Proof ideas}

We now briefly describe (in informal terms) the proof strategy for Theorem \ref{main}, which follows the ideas in \cite{tt}, but now contains some ``Archimedean'' arguments (relating to the Archimedean characters $n \mapsto n^{it}$) in addition to the ``non-Archimedean'' arguments in \cite{tt} (that related to the Dirichlet characters $n \mapsto \chi(n)$).  The new features compared to \cite{tt} include extensive use of the fact that the correlations $f_d(a)$ are ``slowly varying'' in terms of $d$ (this is made precise in formula \eqref{fd1d2}), and the use of this to derive ``approximate quasimorphism properties'' for certain quantities related to these correlations (these are detailed below). We then prove that the approximate quasimorphisms are very close to  actual quasimorphisms (which in our case are Dirichlet characters or Archimedean characters), which eventually leads to the desired conclusions.

As already noted, one key ingredient is (a rigorous form of) the approximate identity \eqref{nxp} that arises from the entropy decrement argument.  In terms of the correlation functions $f_d(a)$, this identity takes the (heuristic) form
$$ f_{dp}(a) G(p) \approx f_{d}(ap)$$
for any integers $a,d$ and ``most'' $p$, where $G \coloneqq g_1 \cdots g_k$; see Proposition \ref{dollop} for a precise statement.  Compared to \cite{tt}, the main new difficulty is the dependence of $f_d$ on the $d$ parameter. 

Assuming for simplicity that $G$ has modulus $1$ (which is the most difficult case), we thus have
$$ f_{dp}(a) \approx f_{d}(ap) \overline{G(p)}$$
for any integers $a,d$ and ``most'' $p$.  Iterating this leads to
\begin{equation}\label{p1p2}
 f_{p_1 p_2}(a) \approx f_1(ap_1p_2) \overline{G(p_1)} \overline{G(p_2)}
\end{equation}
for ``most'' primes $p_1,p_2$ (more precisely, the difference between the two sides of the equation is $o(1)$ when suitably averaged over  $p_1,p_2$; see Corollary \ref{mango}).   On the other hand, results from ergodic theory (such as \cite{leibman}, \cite{le})  give control on the function $f_1(a)$, describing it (up to negligible errors) as a nilsequence, which can then be decomposed further into a periodic piece $f_{1,0}$ and an ``irrational'' component.  The irrational component was already shown in \cite{tt} to give a negligible contribution to the equation \eqref{p1p2} after performing some averaging in $p_1,p_2$, thanks to certain bilinear estimates for nilsequences.  As such, one can effectively replace $f_1$ here by the periodic component $f_{1,0}$ (see \eqref{2m1-0} for a precise statement).

We thus reach the relation
\begin{align*}
 f_{p_1 p_2}(a) \approx f_{1,0}(ap_1p_2) \overline{G(p_1)} \overline{G(p_2)}
 \end{align*}
 for ``most'' $p_1,p_2$. Let $q$ be the period of $f_{1,0}$. If we pick two large primes $p_1\equiv c\pmod {q}$ and $p_1'\equiv bc\pmod{q}$ for arbitrary $b,c\in (\mathbb{Z}/q\mathbb{Z})^{\times}$ with $p_1\approx p_1'$ (using the the prime number theorem), we get
 \begin{align*}
 f_{1,0}(acp_2) \overline{G(p_1)}\approx  f_{1,0}(abcp_2) \overline{G(p_1')},   
 \end{align*}
 for ``most'' $p_1,p_1',p_2$, since the averages $f_d(a)$ are slowly varying as a function of $d$ (see equation \eqref{fd1d2} for the precise meaning of this). Choosing $p_2\equiv 1 \pmod{q}$, we see that the quotient $f_{1,0}(ac)/f_{1,0}(abc)$ is independent of $a$ (since $p_1,p_1'$ were independent of $a$). Substituting then $a=a_1$ and $a=a_2$ to the quotient, we get the approximate identity
 \begin{align}\label{eq7}
f_{1,0}(a_1c)f_{1,0}(a_2bc)\approx f_{1,0}(a_1bc)f_{1,0}(a_2c); 
 \end{align}
 see Proposition \ref{prop1} for a precise version of this, where we need to average over $c$ to make the argument rigorous. We may assume that $f(a_0)\neq 0$ for some $a_0$, as otherwise there is nothing to prove, and this leads to $f_{1,0}(a_0)\neq 0$. Taking $a_1\equiv a_0c^{-1} \pmod{q}$, $a_2\equiv a_0\pmod{q}$ in  \eqref{eq7}, we are led to 
\begin{align*}
f_{1,0}(a_0)f_{1,0}(a_0bc)\approx f_{1,0}(a_0b)f_{1,0}(a_0c)    
\end{align*}
Thus,  the function $\psi(x)=f_{1,0}(a_0x)/f_{1,0}(a_0)$ satisfies the approximate quasimorphism equation
$$ \psi(b_1 b_2) \approx \psi(b_1) \psi(b_2)$$
for $b_1,b_2 \in (\Z/q\Z)^\times$ ranging in the invertible residue classes in $\Z/q\Z$ and some unknown function $\psi: (\Z/q\Z)^\times \to \C$ (to make the above deductions rigorous, we need to take as $\psi(x)$ an averaged version of $x\mapsto f_{1,0}(a_0x)/f_{1,0}(a_0)$). Moreover, the function $\psi(x)$ takes values comparable to $1$.  Of course, Dirichlet characters obey the quasimorphism equation exactly; and we can use standard ``cocycle straightening'' arguments to show conversely that any solution to the quasimorphism equation must be very close to a Dirichlet character $\chi$ (see Lemma \ref{dirs} for a precise statement).  This will be used to show that $f_{1,0}$ and $f_d$ are essentially $\chi$-isotypic.

Once this isotopy property is established, one can then return to \eqref{p1p2} and analyse the dependence of various components of \eqref{p1p2} 
on the Archimedean magnitudes of $p_1,p_2$ rather than their residues mod $q$.  One can eventually transform this equation again to the quasimorphism equation, but this time on the multiplicative group $\R^+$ rather than $(\Z/q\Z)^\times$ (also, the functions $\psi$ will be ``log-Lipschitz'' in a certain sense).  Now it is the Archimedean characters $n \mapsto n^{it}$ that are the model solutions of this equation, and we will again be able to show that all other solutions to this equation are close to an Archimedean character (see Lemma \ref{dira} for a precise statement).  Once one has extracted both the Dirichlet character $\chi$ and the Archimedean character $n \mapsto n^{it}$ in this fashion, the rest of Theorem \ref{main} can be established by some routine calculations.

\subsection{Acknowledgments}

We thank the anonymous referees for careful reading of the paper and for useful comments.

TT was supported by a Simons Investigator grant, the James and Carol Collins Chair, the Mathematical Analysis \&
Application Research Fund Endowment, and by NSF grant DMS-1266164.

JT was supported by UTUGS Graduate School and project number 293876 of the Academy of Finland. He thanks UCLA for excellent working conditions during a visit there in April 2018, during which a large proportion of this work was completed.

\subsection{Notation}\label{subsec: not}

We use the usual asymptotic notation $X \ll Y$, $Y \gg X$, or $X = O(Y)$ to denote the bound $|X| \leq CY$ for some constant $C$.  If $C$ needs to depend on parameters, we will denote this by subscripts, thus for instance $X \ll_k Y$ denotes the estimate $|X| \leq C_k Y$ for some $C_k$ depending on $k$.  We also write $o_{n \to \infty}(Y)$ for a quantity bounded in magnitude by $c(n)Y$ for some $c(n)$ that goes to zero as $n \to \infty$ (holding all other parameters fixed). For any set $\mathcal{X}\subset \mathbb{N}$ with infinite complement, we define the limit operator $\lim_{X\to \infty; X\not \in \mathcal{X}} f(X)$ as $\lim_{n\to \infty} f(x_n)$, where $x_1,x_2,\ldots$ are the elements of the complement $\mathbb{N}\setminus \mathcal{X}$ in strictly increasing order.

We use $a\ (q)$ to denote the residue class of $a$ modulo $q$.  If $E$ is a set, we write $1_E$ for its indicator function, thus $1_E(n)=1$ when $n \in E$ and $1_E(n)=0$ otherwise. 

We use the following standard multiplicative functions throughout the paper:
\begin{itemize}
\item The \emph{Liouville function} $\lambda$, which is the $1$-bounded completely multiplicative function with $\lambda(p)=-1$ for all primes $p$;
\item The \emph{M\"obius function} $\mu$, which is equal to $\lambda$ at square-free numbers and $0$ elsewhere.
\item \emph{Dirichlet characters} $\chi$, which are $1$-bounded completely multiplicative functions of some period $q$, with $\chi(n)$ non-zero precisely when $n$ is coprime to $q$; and
\item \emph{Archimedean characters} $n \mapsto n^{it}$, where $t$ is a real number.

\item \emph{twisted Dirichlet characters} $n \mapsto \chi(n)n^{it}$, which are the product of a Dirichlet character and an Archimedean character.
\end{itemize}

In the arguments that follow, asymptotic averages of various types feature frequently, so we introduce some abbreviations for them.

\begin{definition}[Asymptotic averaging notation]
If $f: \N \to \C$ is a function, we define the asymptotic average
$$ \E_{n \in \N} f(n) \coloneqq \lim_{X \to \infty} \E_{n \leq X} f(n) $$
provided that the limit exists.  We adopt the convention that assertions such as $\E_{n \in \N} f(n) = \alpha$ are automatically false if the limit involved does not exist.  Similarly define $\E_{n \in \N}^{\log} f(n)$ and $\E_{n \in \N}^{\log\log} f(n)$.  If $f: {\P} \to \C$ is a function, we similarly define
$$ \E_{p \in \P} f(p) \coloneqq \lim_{X \to \infty} \E_{p \leq X} f(p) $$
and
$$ \E_{p \in \P}^{\log} f(p) \coloneqq \lim_{X \to \infty} \E_{p \leq X}^{\log} f(p).$$
Moreover, given a generalised limit functional $\lim^*_{X \to \infty}$, we define the corresponding asymptotic limits $\E^*_{n \in \N}$, $\E^{\log,*}_{n \in \N}$, $\E^{\log\log,*}_{n \in \N}$, $\E^*_{p \in {\P}}$, $\E^{\log,*}_{p \in {\P}}$ by replacing the ordinary limit functional by the generalised limit, thus for instance
$$ \E^{\log,*}_{n \in \N} f(n) \coloneqq \lim^*_{X \to \infty} \E^{\log}_{n \leq X} f(n).$$
If an ordinary asymptotic limit such as $\E^{\log}_{n \in \N} f(n)$ exists, then $\E^{\log,*}_{n \in \N} f(n)$ will attain the same value; but the latter limit exists for all bounded sequences $f$, whereas the ordinary limit need not exist.  In later parts of the paper we will also need an additional generalised limit $\lim^{**}_{X \to \infty}$, and one can then define generalised asymptotic averages such as $\E^{\log,**}_{n \in \N} f(n)$ accordingly. 
\end{definition}

\begin{remark}\label{avg-com}  If $f$ is a bounded sequence and $\alpha$ is a complex number, a standard summation by parts exercise shows that the statement $\E_{n \in \N} f(n) = \alpha$ implies $\E_{n \in \N}^{\log} f(n) = \alpha$, which in turn implies $\E_{n \in \N}^{\log\log} f(n) = \alpha$, and similarly $\E_{p \in \P} f(p) = \alpha$ implies $\E_{p \in {\P}}^{\log} f(p) = \alpha$; however, the converse implications can be highly non-trivial or even false.  For instance, as mentioned earlier,  it is not difficult to show that $\E_{n \in \N}^{\log} n^{it} = 0$ for any $t\neq 0$, but the limit $\E_{n \in \N} n^{it}$ does not exist.  (On the other hand, from the prime number theorem and partial summation one has $\E_{p \in \P} p^{it} = 0$.) In the same spirit, if $A$ is the set of integers whose decimal expansion has leading digit $1$, then one easily computes ``Benford's law'' $\E_{n\in \N}^{\log}1_{A}(n)=(\log 2)/(\log 10)$, whereas $\E_{n\in \N} 1_{A}(n)$ fails to exist.
\end{remark}

\section{Proof of main theorem}\label{main-sec}

In this section we establish Theorem \ref{main}.  We first establish a version of the Furstenberg correspondence principle.

\begin{proposition}[Furstenberg correspondence principle]\label{corr} Let the notation and hypotheses be as in Theorem \ref{main}.  Then for any real number $d > 0$, there exist random functions $\mathbf{g}^{(d)}_1,\dots,\mathbf{g}^{(d)}_k: \Z \to \Disk$ and a random profinite integer\footnote{The \emph{profinite integers} $\hat \Z$ are the inverse limit of the cyclic groups $\Z/q\Z$, with the weakest topology that makes the reduction maps  $n \mapsto n\ (q)$ continuous.  This is a compact abelian group and therefore it has a well-defined probability Haar measure.} $\mathbf{n}^{(d)} \in \hat \Z$, all defined on a common probability space $\Omega^{(d)}$, such that
$$ \E^{(d)} F( ((\mathbf{g}^{(d)}_i(h))_{1 \leq i \leq k, -N \leq h \leq N}, \mathbf{n}^{(d)}\ (q))) = \lim^{*}_{X \to \infty} \E_{n \leq X/d} F((g_i(n+h))_{1 \leq i \leq k, -N \leq h \leq N}, n\ (q))$$
for any natural numbers $N,q$ and any continuous function $F: \Disk^{k(2N+1)}\times \mathbb{Z}/q\mathbb{Z} \to \C$, where $\E^{(d)}$ denotes the expectation on the probability space $\Omega^{(d)}$.  Furthermore, the random variables $\mathbf{g}^{(d)}_1,\dots,\mathbf{g}^{(d)}_k: \Z \to \Disk$ and $\mathbf{n}^{(d)} \in \hat \Z$ are a stationary process, by which we mean that for any natural number $N$, the joint distribution of $(\mathbf{g}^{(d)}_i(n+h))_{1 \leq i \leq k, -N \leq h \leq N}$ and $\mathbf{n}^{(d)}+n$ does not depend on $n$ as $n$ ranges over the integers.
\end{proposition}

\begin{proof}  Up to some minor notational changes, this is essentially \cite[Proposition 3.1]{tt}, applied once for each value of $d$. The only difference is that the logarithmic averaging $\E^{\log}_{x_m/w_m \leq n \leq x_m}$ there has been replaced by the non-logarithmic averaging $\E_{n \leq X/d}$.  However, an inspection of the arguments reveal that the proof of the proposition is essentially unaffected by this change.
\end{proof}

Let $G: \N \to \Disk$ denote the multiplicative function $G \coloneqq g_1 \cdots g_k$.  We now adapt the entropy decrement arguments from \cite[\S 3]{tt} to establish the approximate relation
\begin{equation}\label{fdap}
 f_d(ap) \approx f_{dp}(a) G(p)
\end{equation}
for integers $a$, real numbers $d>0$, and ``most'' primes $p$.

Fix $a,d$, and let $p$ be a prime.  From \eqref{fda-def} we have
$$
 f_{dp}(a) G(p) = \lim^{*}_{x \to \infty} \E_{n \leq x/dp} g_1(p) g_1(n+ah_1) \cdots g_k(p) g_k(n+ah_k).
$$
From multiplicativity, we can write $g_j(p) g_j(n+ah_j)$ as $g_j(pn + aph_j)$ unless $n = -ah_j\ (p)$.   The latter case contributes $O\left(\frac{1}{p}\right)$ to the above limit (where we allow implied constants to depend on $k$), thus
$$
 f_{dp}(a) G(p) = \lim^{*}_{x \to \infty} \E_{n \leq x/dp} g_1(pn+aph_1) \cdots g_k(pn+aph_k) + O\left( \frac{1}{p} \right).$$
If we now make $pn$ rather than $n$ the variable of summation, we conclude that
$$
 f_{dp}(a) G(p) = \lim^{*}_{x \to \infty} \E_{n \leq x/d} g_1(n+aph_1) \cdots g_k(n+aph_k) p 1_{p|n} + O\left( \frac{1}{p} \right).$$
Comparing this with \eqref{fda-def}, we conclude that
$$
 f_{dp}(a) G(p) - f_d(ap) = \lim^{*}_{x \to \infty} \E_{n \leq x/d} g_1(n+aph_1) \cdots g_k(n+aph_k) (p 1_{p|n}-1) + O\left( \frac{1}{p} \right)$$
and hence by Proposition \ref{corr}
\begin{equation}\label{fdag}
f_{dp}(a) G(p) - f_d(ap) = \E^{(d)} \mathbf{g}^{(d)}_1(aph_1) \cdots \mathbf{g}^{(d)}_k(aph_k) (p 1_{p|\mathbf{n}^{(d)}}-1) + O\left( \frac{1}{p} \right).
\end{equation}

On the other hand, by repeating the proof of \cite[Theorem 3.6]{tt} verbatim (see also \cite[Remark 3.7]{tt}), we have the following general estimate:

\begin{proposition}[Entropy decrement argument]\label{eda}
Let  $\mathbf{g}_1,\dots,\mathbf{g}_k: \Z \to \Disk$ be random functions and $\mathbf{n} \in \hat \Z$ be a stationary process, let $a,h_1,\dots,h_k$ be integers, and let $0 < \eps < \frac{1}{2}$ be real.  Then one has
$$ \E_{2^m \leq p < 2^{m+1}} |\mathbb{E}\mathbf{g}_1(aph_1) \cdots \mathbf{g}_k(aph_k) (p 1_{p|\mathbf{n}}-1) | \leq \eps$$
for all natural numbers $m$ outside of an exceptional set ${\mathcal M}$ obeying the bound
\begin{equation}\label{eps}
 \sum_{m \in {\mathcal M}} \frac{1}{m} \ll_{a, h_1,\dots,h_k} \eps^{-4} \log \frac{1}{\eps}.
\end{equation}
\end{proposition}

Note that the bound \eqref{eps} is uniform in the random functions $\mathbf{g}_1,\dots,\mathbf{g}_k$ (although the set $\mathcal{M}$ may depend on these functions). Summing the result over different dyadic scales gives us the following version of \eqref{fdap}.

\begin{proposition}[Approximate isotopy]\label{dollop}  Let the notation and hypotheses be as in Theorem \ref{main}.  Let $a$ be an integer, and let $\eps > 0$ be real.  Then for sufficiently large $P$, we have
$$ \sup_{d >0} \E^{\log}_{p \leq P} |f_{dp}(a) G(p) - f_d(ap)| \leq \eps$$
where the supremum is over positive reals.
\end{proposition}

A key technical point for our application is that while $P$ may depend on $a,\eps$, it can be taken to be uniform in $d$.

\begin{proof}   Let $a,\eps,P$ be as in the proposition, and let $d >0$. We may assume that $\varepsilon>0$ is small.  By the prime number theorem, we have
$$ \E^{\log}_{p \leq P} |f_{dp}(a) G(p) - f_d(ap)|  \ll \E^{\log}_{m \leq (\log P)/(\log 2)} \E_{2^m \leq p < 2^{m+1}} |f_{dp}(a) G(p) - f_d(ap)|.$$
By \eqref{fdag} and Proposition \ref{eda}, we have
$$ \E_{2^m \leq p < 2^{m+1}} |f_{dp}(a) G(p) - f_d(ap)| \leq \eps^2$$
for all $m$ outside of an exceptional set ${\mathcal M}_{a,\eps,d}$ obeying the bound
$$ \sum_{m \in {\mathcal M}_{a,\eps,d}} \frac{1}{m} \ll_{a, h_1,\dots,h_k} \eps^{-8} \log \frac{1}{\eps}.$$
In the exceptional set ${\mathcal M}_{a,\eps,d}$, we use the trivial bound
$$ \E_{2^m \leq p < 2^{m+1}} |f_{dp}(a) G(p) - f_d(ap)| \ll 1$$
to conclude that
$$ \E^{\log}_{p \leq P} |f_{dp}(a) G(p) - f_d(ap)| \ll \eps^2 + O_{a,h_1,\dots,h_k}\left( \frac{\eps^{-8} \log \frac{1}{\eps}}{\log\log P} \right),$$
and the claim follows by choosing $P$ large in terms of $a,\varepsilon, h_1,\ldots, h_k$.
\end{proof}

As in \cite{tt}, we iterate this approximate formula to obtain

\begin{corollary}\label{mango}  For any integer $a$ one has
$$ \limsup_{P_1 \to \infty} \limsup_{P_2 \to \infty} \E_{p_1 \leq P_1}^{\log} \E_{p_2 \leq P_2}^{\log} |f_{p_1 p_2}(a) G(p_1) G(p_2) - f_1(ap_1 p_2)| = 0$$
\end{corollary}

\begin{proof}  Let $a$ be an integer, let $\eps>0$ be real, let $P_1$ be sufficiently large depending on $a,\eps$, and let $P_2$ be sufficiently large depending on $a,\eps,P_1$.  From Proposition \ref{dollop} one has
$$ \E^{\log}_{p_1 \leq P_1} |f_{p_1 p_2}(a) G(p_1) - f_{p_2}(ap_1)| \ll \eps$$
for all primes $p_2$, and hence
$$ \E_{p_1 \leq P_1}^{\log} \E_{p_2 \leq P_2}^{\log} |f_{p_1 p_2}(a) G(p_1) G(p_2) - f_{p_2}(ap_1) G(p_2)| \ll \eps.$$
On the other hand, from a second application of Proposition \ref{dollop} one has
$$ \E^{\log}_{p_2 \leq P_2} |f_{p_2}(ap_1) G(p_2) - f_{1}(ap_1 p_2)| \ll \eps$$
for all $p_1 \leq P_1$, and hence
$$ \E_{p_1 \leq P_1}^{\log} \E_{p_2 \leq P_2}^{\log} |f_{p_2}(ap_1) G(p_2) - f_1(ap_1 p_2)| \ll \eps.$$
From the triangle inequality we thus have
$$ \E_{p_1 \leq P_1}^{\log} \E_{p_2 \leq P_2}^{\log} |f_{p_1 p_2}(a) G(p_1) G(p_2) - f_1(ap_1 p_2)| \ll \eps$$
under the stated hypotheses on $\eps,P_1,P_2$.  Taking limit superior in $P_2$ and then in $P_1$, we conclude that
$$ \limsup_{P_1 \to \infty} \limsup_{P_2 \to \infty} \E_{p_1 \leq P_1}^{\log} \E_{p_2 \leq P_2}^{\log} |f_{p_1 p_2}(a) G(p_1) G(p_2) - f_1(ap_1 p_2)| \ll \eps$$
for any $\eps>0$, and the claim follows.
\end{proof}

Next, we have the following structural description of $f_1$.

\begin{proposition}\label{pe}  Let $f_1$ be as in Theorem \ref{main}. For any $\eps>0$, one can write
$$ f_{1} = f_{1,0} + g$$
where $f_{1,0} = f_{1,0}^{(\eps)}$ is periodic, and the error $g = g^{(\eps)}$ obeys the bilinear estimate
\begin{equation}\label{unweighted}
 \E_{p_1 \leq x} \E_{p_2 \leq y} \alpha_{p_1} \beta_{p_2} g(a p_1 p_2) \ll \eps 
\end{equation}
as well as the logarithmic counterpart
$$ \E^{\log}_{p_1 \leq x} \E^{\log}_{p_2 \leq y} \alpha_{p_1} \beta_{p_2} g(a p_1 p_2) \ll \eps $$
whenever $a$ is a non-zero integer, $x$ is sufficiently large depending on $a,\eps$; $y$ is sufficiently large depending on $x,a,\eps$; and $\alpha_{p_1}, \beta_{p_2} = O(1)$ are bounded sequences.  
\end{proposition}

\begin{proof}  We freely use the notation from \cite[Sections 4--5]{tt}.  By summation by parts it suffices to obtain a decomposition obeying \eqref{unweighted}.
By repeating the proof of \cite[Corollary 4.6]{tt} verbatim\footnote{In \cite[Corollary 4.6]{tt}, $a$ was required to be a natural number rather than a non-zero integer, however one can easily adapt the arguments to the case of negative $a$ with only minor modifications (in particular, one has to modify the definition of $\mathbf{X}_m$ slightly to allow $l$ to be negative).}, we can write
$$ f_1 = f_{1,1} + f_{1,2}$$
where $f_{1,1}$ is a nilsequence of some finite degree $D$, and $f_{1,2}$ obeys the asymptotic
$$ \lim_{x \to \infty} \E_{p \leq x} |f_{1,2}(ap)| = 0$$
for any non-zero integer $a$.  We can now neglect the $f_{1,2}$ term as it can be absorbed into the $g$ error.  Next, applying \cite[Proposition 5.6]{tt}, we can decompose
$$ f_{1,1} = f_{1,0} + \sum_{i=1}^D \sum_{j=1}^{J_i} c_{i,j} \chiup_{i,j}$$
for some periodic function $f_{1,0}$, some non-negative integers $J_1,\dots,J_D$, some irrational nilcharacters $\chiup_{i,j}$ of degree $i$, and some linear functionals $c_{i,j}$.  Using \cite[Lemma 5.8]{tt} (noting that if $\chiup$ is an irrational nilcharacter, then so is $\chiup(a \cdot)$) we see that each of the terms $c_{i,j} \chiup_{i,j}$ can be absorbed into the error term $g$.  The claim then follows from the triangle inequality.  
\end{proof}

Finally, we record a simple log-Lipschitz estimate
\begin{equation}\label{fd1d2}
|f_{d_1}(a) - f_{d_2}(a)| \leq 2|\log d_1 - \log d_2|
\end{equation}
for any integer $a$ and any real $d_1,d_2 > 0$; this follows by using \eqref{fda-def} and the triangle inequality to estimate $|f_{d_1}(a)-f_{d_2}(a)|\leq \frac{2|d_1-d_2|}{\max\{d_1,d_2\}}$ and then the mean value theorem to $x\mapsto \log x$.

We return to the proof of Theorem \ref{main}.  If we have
$$  \limsup_{X \to \infty} \E^{\log\log}_{d \leq X} |f_d(a)| = 0$$
for all $a$, then the claim follows by setting $f=0$, so we may assume without loss of generality that there exists an integer $a_0$ such that
$$ \limsup_{X \to \infty} \E^{\log\log}_{d \leq X} |f_d(a_0)| > 0.$$
Thus, by the Hahn-Banach theorem, we may find a generalised limit $\lim^{**}_{X \to \infty}$ (which may or may not be equal to the previous generalised limit $\lim_{X \to \infty}^*$) such that
$$ \lim^{**}_{X \to \infty} \E^{\log\log}_{d \leq X} |f_d(a_0)| > 0,$$
and thus using the generalised limit asymptotic notation associated to $\lim^{**}_{X \to \infty}$ (see Subsection \ref{subsec: not}), we have
\begin{equation}\label{fda}
 \E^{\log\log,**}_{d \in \N} |f_d(a_0)| \gg 1.
\end{equation}

For future reference we record the following convenient lemma relating the averaging operator $\E^{\log\log,**}_{d \in \N}$ with $\E^{\log,**}_{p \in {\P}}$:

\begin{lemma}[Comparing averages over integers and primes]\label{chug}  Let $f: \N \to \C$ be a function which is bounded log-Lipschitz in the sense that there is a constant $C$ such that $|f(d)| \leq C$ and $|f(d)-f(d')| \leq C |\log d - \log d'|$ for all $d,d' \in \N$.  Then for any natural number $a$, one has
$$\limsup_{X\to \infty}|\E^{\log\log}_{d \leq X} f(d)-\E^{\log}_{p \leq X} f(ap)|=0,$$
so in particular
$$ \E^{\log\log,**}_{d \in \N} f(d) = \E^{\log,**}_{p \in {\P}} f(ap).$$
\end{lemma}

\begin{proof}  We allow implied constants to depend on $C,a$.  Let $\eps>0$, and assume $X$ is sufficiently large depending on $C,\eps$.  Then from the prime number theorem and the bounded log-Lipschitz property we have
\begin{align*}
\E^{\log}_{p \leq X} f(ap) &= \frac{1}{\log\log X} \sum_{p \leq X} \frac{f(ap)}{p} + O(\eps) \\
&= \frac{1}{\log\log X} \sum_{d \leq X} \frac{1}{\eps d} \sum_{d \leq p \leq (1+\eps) d} \frac{f(ap)}{p} + O(\eps) \\
&= \frac{1}{\log\log X} \sum_{d \leq X} \frac{1}{\eps d} \sum_{d \leq p \leq (1+\eps) d} \frac{f(ad)}{d} + O(\eps) \\
&= \frac{1}{\log\log X} \sum_{d \leq X} \frac{f(ad)}{d \log(2+d)} + O(\eps). 
\end{align*}
Again by the bounded log-Lipschitz property, we have
\begin{align*}
f(ad)=\frac{1}{a}\sum_{ad\leq d'< a(d+1)}f(d')+O(1/d),    
\end{align*}
and inserting this into the preceding computation, we get
\begin{align*}
\E^{\log}_{p \leq X} f(ap)&=\frac{1}{\log\log X} \sum_{d' \leq aX} f(d') \cdot\frac{1}{a}\sum_{d'/a-1< d\leq d'/a}\frac{1}{d\log (2+d)}+O(\varepsilon)\\
&=\frac{1}{\log\log X} \sum_{d' \leq X} \frac{f(d')}{d'\log(2+d')}+O(\varepsilon).
\end{align*}
Taking the absolute value of the difference of the two sides of this equation, applying $\limsup_{X \to \infty}$ and then sending $\eps \to 0$, we obtain the claim.
\end{proof}

Now, let $\eps>0$ be a sufficiently small parameter. If one had
$$ \sum_p \frac{1-|g_j(p)|}{p} = \infty$$
for some $1 \leq j \leq k$, then by Wirsing's theorem \cite{wirsing} as in \cite[\S 6]{tt} one would have $f_d(a)=0$ for all $a,d$.  Thus we may assume that
$$ \sum_p \frac{1-|g_j(p)|}{p} < \infty$$
for all $j$, which implies in particular that one has 
\begin{equation}\label{gp-inv}
1-\varepsilon\leq |G(p)|\leq 1
\end{equation}
for all but finitely many $p$.  For any integer $a$, we see from Corollary \ref{mango} that
$$ \limsup_{P_1 \to \infty} \limsup_{P_2 \to \infty} \E_{p_1 \leq P_1}^{\log} \E_{p_2 \leq P_2}^{\log} |f_{p_1 p_2}(a) G(p_1) G(p_2) - f_1(ap_1 p_2)| \ll \eps.$$
By \eqref{gp-inv} we then have
$$ \limsup_{P_1 \to \infty} \limsup_{P_2 \to \infty} \E_{p_1 \leq P_1}^{\log} \E_{p_2 \leq P_2}^{\log} |f_{p_1 p_2}(a) - \overline{G(p_1)} \overline{G(p_2)}  f(ap_1 p_2)| \ll \eps$$
Applying Proposition \ref{pe}, we conclude that 
\begin{equation}\label{2m1-0}
\limsup_{P_1 \to \infty} \limsup_{P_2 \to \infty} \E_{p_1 \leq P_1}^{\log} \E_{p_2 \leq P_2}^{\log} |f_{p_1 p_2}(a) - \overline{G(p_1)} \overline{G(p_2)} f_{1,0}(ap_1 p_2)| \ll \eps.
\end{equation}
In particular we have
\begin{equation}\label{2m1}
 \E_{p_1 \in {\P}}^{\log,**} \E_{p_2 \in {\P}}^{\log,**} |f_{p_1 p_2}(a) - \overline{G(p_1)} \overline{G(p_2)} f_{1,0}(ap_1 p_2)| \ll \eps.
\end{equation}

Heuristically, \eqref{2m1} asserts the approximation
\begin{equation}\label{2m1-approx}
 f_{p_1 p_2}(a) \approx \overline{G(p_1)} \overline{G(p_2)} f_{1,0}(ap_1 p_2)
\end{equation}
for ``most'' $a,p_1,p_2$.  This turns out to be a remarkably powerful approximate equation, giving a lot of control on the functions $G$, $f_d$,  and $f_{1,0}$.  Roughly speaking, we will be able to show that the only way to solve \eqref{2m1-approx} (in a manner compatible with \eqref{fda} and \eqref{fd1d2}) is if $G(p) \approx \chi(p) p^{it}$, $f_d(a) \approx f(a) d^{-it}$, and $f_{1,0} \approx f$ for some $\chi$-isotypic $q$-periodic function $f$.  Conversely, it is easy to see that if $G, f_d, f_{1,0}$ are of the above form, then they obey \eqref{2m1-approx}.

We first use \eqref{2m1} to control $f_{1,0}$.  Let $q$ denote the period of $f_{1,0}$ (which depends on $\varepsilon$); by abuse of notation, we view $f_{1,0}$ as a function on $\Z/q\Z$ as well as on $\Z$.  We then have

\begin{proposition}[Initial control on $f_{1,0}$]\label{prop1}  Let $a_0$ be as in \eqref{fda}. We have
\begin{equation}\label{f10-large}
\E_{c \in (\Z/q\Z)^\times} |f_{1,0}( a_0 c)| \gg 1.
\end{equation}
Furthermore, for any integers $a_1,a_2$ and any natural number $b$ coprime to $q$, we have
\begin{equation}\label{fito}
\E_{c \in (\Z/q\Z)^\times} |f_{1,0}(a_1 c)f_{1,0}(a_2bc) - f_{1,0}(a_1 bc) f_{1,0}(a_2 c)| \ll \eps.
\end{equation}
\end{proposition}

\begin{proof}  By Lemma \ref{chug}, \eqref{fda} and \eqref{fd1d2}, we see that
$$\E_{p_2 \in {\P}}^{\log,**} |f_{p_1 p_2}(a_0)| = \E_{d \in \N}^{\log\log,**} |f_{d}(a_0)| \gg 1$$
for any $p_1$, and hence
$$ \E_{p_1 \in {\P}}^{\log,**} \E_{p_2 \in {\P}}^{\log,**} |f_{p_1 p_2}(a_0)| \gg 1.$$
On the other hand, from \eqref{2m1} we have
\begin{equation}\label{m1}
 \E_{p_1 \in {\P}}^{\log,**} \E_{p_2 \in {\P}}^{\log,**} |f_{p_1 p_2}(a_0) - \overline{G(p_1)} \overline{G(p_2)} f_{1,0}(a_0 p_1 p_2)| \ll \eps.
\end{equation}
From the triangle inequality, we have
$$ |f_{p_1 p_2}(a_0)| \ll |f_{1,0}(a_0 p_1 p_2)| + |f_{p_1 p_2}(a_0) - \overline{G(p_1)} \overline{G(p_2)} f_{1,0}(a_0 p_1 p_2)|+\varepsilon,$$
and hence (since $\eps$ is assumed small)
$$  \E_{p_1 \in {\P}}^{\log,**} \E_{p_2 \in {\P}}^{\log,**} |f_{1,0}(a_0 p_1 p_2)| \gg 1.$$
By the periodicity of $f_{1,0}$ and the prime number theorem in arithmetic progressions, we conclude \eqref{f10-large}.

Next, let $a_1,a_2,b$ be as in the proposition.  Applying \eqref{2m1} twice, we see that 
\begin{equation}\label{mas-1}
\E_{p_1 \in {\P}}^{\log,**} \E_{p_2 \in {\P}}^{\log,**} |f_{p_1 p_2}(a_1) - \overline{G(p_1)} \overline{G(p_2)} f_{1,0}(a_1 p_1 p_2)| \ll \eps
\end{equation}
and 
\begin{equation}\label{mas-2}
\E_{p_1 \in {\P}}^{\log,**} \E_{p_2 \in {\P}}^{\log,**}  |f_{p_1 p_2}(a_2) - \overline{G(p_1)} \overline{G(p_2)} f_{1,0}(a_2 p_1 p_2)| \ll \eps.
\end{equation}

We now eliminate the functions $f_{p_1p_2}$ and $G$ from these estimates.  As in the proof of Lemma \ref{chug}, we can use the prime number theorem in arithmetic progressions to rearrange the left-hand side of \eqref{mas-1} as 
$$
 \E_{p_1 \in {\P}}^{\log,**} \E_{c \in (\Z/q\Z)^\times} \E_{d \in \N}^{\log\log,**} \E_{d \leq p_2 < (1+\eps) d; p_2 = c\ (q)} |f_{p_1 p_2}(a_1) - \overline{G(p_1)} \overline{G(p_2)} f_{1,0}(a_1 p_1 p_2)| + O(\eps)$$
and hence after a change of variables $c \mapsto bc$ (and renaming $p_2$ as $p'_2$)
$$
 \E_{p_1 \in {\P}}^{\log,**} \E_{c \in (\Z/q\Z)^\times} \E_{d \in \N}^{\log\log,**} \E_{d \leq p'_2 < (1+\eps) d; p'_2 = bc\ (q)} 
 |f_{p_1 p'_2}(a_1) - \overline{G(p_1)} \overline{G(p'_2)} f_{1,0}(a_1 p_1 p'_2)| \ll \eps.$$
From \eqref{fd1d2}, we have $f_{p_1 p_2}(a_1), f_{p_1 p'_2}(a_1) = f_{p_1 d}(a_1) + O(\eps)$; from the periodicity of $f_{1,0}$ we also have $f_{1,0}(a_1 p_1 p_2) = f_{1,0}(a_1 c p_1)$ and $f_{1,0}(a_1 p_1 p'_2) = f_{1,0}(a_1bc p_1)$.  We conclude that
$$
 \E_{p_1 \in {\P}}^{\log,**} \E_{c \in (\Z/q\Z)^\times} \E_{d \in \N}^{\log\log,**} \E_{d \leq p_2 < (1+\eps) d; p_2 = c\ (q)}  |f_{p_1 d}(a_1) - \overline{G(p_1)} \overline{G(p_2)} f_{1,0}(a_1 cp_1)| \ll \eps$$
and
$$
 \E_{p_1 \in {\P}}^{\log,**} \E_{c \in (\Z/q\Z)^\times} \E_{d \in \N}^{\log\log,**} \E_{d \leq p'_2 < (1+\eps) d; p'_2 = bc\ (q)} |f_{p_1 d}(a_1) - \overline{G(p_1)} \overline{G(p'_2)} f_{1,0}(a_1bcp_1)| \ll \eps$$
and hence by the triangle inequality and \eqref{gp-inv} we have
\begin{align*} 
 &\E_{p_1 \in {\P}}^{\log,**} \E_{c \in (\Z/q\Z)^\times} \E_{d \in \N}^{\log\log,**} \E_{d \leq p_2 < (1+\eps) d; p_2 = c\ (q)}  \\
 &\quad \E_{d \leq p'_2 < (1+\eps) d; p'_2 = bc\ (q)}|\overline{G(p_2)} f_{1,0}(a_1 cp_1) - \overline{G(p'_2)} f_{1,0}(a_1bcp_1)| \ll \eps.
 \end{align*}
We have thus eliminated $f_{p_1 p_2}$ and one factor of $G$; we still seek to eliminate the other factor of $G$.  To do this, we replace $a_1$ by $a_2$ in the above analysis to obtain
\begin{align*}
 &\E_{p_1 \in {\P}}^{\log,**} \E_{c \in (\Z/q\Z)^\times} \E_{d \in \N}^{\log\log,**} \E_{d \leq p_2 < (1+\eps) d; p_2 = c\ (q)}  \\
 &\quad \E_{d \leq p'_2 < (1+\eps) d; p'_2 = bc\ (q)}|\overline{G(p_2)} f_{1,0}(a_2 cp_1) - \overline{G(p'_2)} f_{1,0}(a_2bcp_1)| \ll \eps.
 \end{align*}
 
 At this point, let us note that $|f_{1,0}(a)|\ll 1$ for $a\in \mathbb{Z}$. To see this, we use Corollary \ref{mango} to conclude that
 \begin{align*}
 \limsup_{P_1 \to \infty} \limsup_{P_2 \to \infty} \E_{p_1 \leq P_1; p_1\equiv 1\ (q)}^{\log} \E_{p_2 \leq P_2; p_2\equiv 1\ (q)}^{\log} |f_{p_1 p_2}(a) - \overline{G(p_1)} \overline{G(p_2)} f_{1,0}(a)| =0.
 \end{align*}
 Then from the triangle inequality, \eqref{gp-inv}, and the trivial bound $|f_{p_1p_2}(a)|\ll 1$ we reach the conclusion $|f_{1,0}(a)|\ll 1$.

Next observe the identity
\begin{align*}
&\overline{G(p_2)} (f_{1,0}(a_1 cp_1)f_{1,0}(a_2bcp_1) - f_{1,0}(a_1bcp_1) f_{1,0}(a_2 cp_1))\\
&= f_{1,0}(a_2 bcp_1) (\overline{G(p_2)} f_{1,0}(a_1 cp_1) - \overline{G(p'_2)} f_{1,0}(a_1bcp_1))\\
&- f_{1,0}(a_1 bcp_1) (\overline{G(p_2)} f_{1,0}(a_2 cp_1) - \overline{G(p'_2)} f_{1,0}(a_2bcp_1));
\end{align*} 
we thus have from the triangle inequality, the boundedness of $|f_{1,0}(a)|$, and \eqref{gp-inv} that
\begin{align*} 
&|f_{1,0}(a_1 cp_1)f_{1,0}(a_2bcp_1) - f_{1,0}(a_1bcp_1) f_{1,0}(a_2 cp_1)| \\
&\ll |\overline{G(p_2)} f_{1,0}(a_1 cp_1) - \overline{G(p'_2)} f_{1,0}(a_1bcp_1)| + |\overline{G(p_2)} f_{1,0}(a_2 cp_1) - \overline{G(p'_2)} f_{1,0}(a_2bcp_1)|
\end{align*}
for all but finitely many $p_1,p_2$, and thus by further application of the triangle inequality
\begin{align*}
& \E_{p_1 \in {\P}}^{\log,**} \E_{c \in (\Z/q\Z)^\times} \E_{d \in \N}^{\log\log,**}\E_{d \leq p_2 < (1+\eps) d; p_2 = c\ (q)} \\
&\quad  \E_{d \leq p'_2 < (1+\eps) d; p'_2 = bc\ (q)} |f_{1,0}(a_1 cp_1)f_{1,0}(a_2bcp_1) - f_{1,0}(a_1bcp_1) f_{1,0}(a_2 cp_1)| \ll \eps.
\end{align*}
As the expression being averaged does not depend on $d,p_2,p'_2$, this bound simplifies to
$$
 \E_{p_1 \in {\P}}^{\log,**} \E_{c \in (\Z/q\Z)^\times} |f_{1,0}(a_1 cp_1)f_{1,0}(a_2bcp_1) - f_{1,0}(a_1bcp_1) f_{1,0}(a_2 cp_1)| \ll \eps$$
and by the prime number theorem in arithmetic progressions and the periodicity of $f_{1,0}$, this simplifies further (cf. Lemma \ref{chug}) to
give the desired bound \eqref{fito}.
\end{proof}

Let $a$ be an integer, and let $b$ be coprime to $q$.  Applying \eqref{fito} with $a_1 = a$ and $a_2 = a_0 c'$ for $c'$ coprime to $q$, and averaging, we conclude that
$$ \E_{c' \in (\Z/q\Z)^\times} \E_{c \in (\Z/q\Z)^\times} |f_{1,0}(a c)f_{1,0}(a_0bc c') - f_{1,0}(a bc) f_{1,0}(a_0 c c')| \ll \eps$$
and hence
$$ \E_{c' \in (\Z/q\Z)^\times} \E_{c \in (\Z/q\Z)^\times} \left|f_{1,0}(a c)f_{1,0}(a_0bc c') \overline{f_{1,0}(a_0cc')}- f_{1,0}(a bc) |f_{1,0}(a_0 c c')|^2 \right| \ll \eps.$$
By the triangle inequality, this implies that
$$ \E_{c \in (\Z/q\Z)^\times} \left|f_{1,0}(a c) \E_{c' \in (\Z/q\Z)^\times} f_{1,0}(a_0bc c') \overline{f_{1,0}(a_0cc')} - f_{1,0}(a bc) \E_{c' \in (\Z/q\Z)^\times} |f_{1,0}(a_0 c c')|^2 \right| \ll \eps.$$
Making the change of variables $c'' = cc'$, this is
$$ \E_{c \in (\Z/q\Z)^\times} \left|f_{1,0}(a c) \E_{c'' \in (\Z/q\Z)^\times} f_{1,0}(a_0bc'') \overline{f_{1,0}(a_0c'')} - f_{1,0}(a bc) \E_{c'' \in (\Z/q\Z)^\times} |f_{1,0}(a_0c'')|^2 \right| \ll \eps.$$
If we define the function $\psi: (\Z/q\Z)^\times \to \C$ by
$$ \psi(b) \coloneqq \frac{\E_{c'' \in (\Z/q\Z)^\times} f_{1,0}(a_0bc'') \overline{f_{1,0}(a_0c'')}}{\E_{c'' \in (\Z/q\Z)^\times} |f_{1,0}(a_0c'')|^2}$$
then by \eqref{f10-large} and Cauchy-Schwarz, we have $\psi(b) = O(1)$ for all $b \in (\Z/q\Z)^\times$, and
\begin{equation}\label{ecc}
 \E_{c \in (\Z/q\Z)^\times} |f_{1,0}(a c) \psi(b) - f_{1,0}(a bc)| \ll \eps
\end{equation}
for all $a \in \Z/q\Z$ and $b \in (\Z/q\Z)^\times$.  

By definition, $\psi(1)=1$.  Next, we establish an approximate multiplicativity property of $\psi$, known as the \emph{quasimorphism} property \cite{quasi} in the literature.  If $b_1, b_2 \in (\Z/q\Z)^\times$, then from three applications of \eqref{ecc} one has
\begin{align*}
 \E_{c \in (\Z/q\Z)^\times} |f_{1,0}(a_0 c) \psi(b_1) - f_{1,0}(a_0 b_1c)| &\ll \eps \\
 \E_{c \in (\Z/q\Z)^\times} |f_{1,0}(a_0 b_1 c) \psi(b_2) - f_{1,0}(a_0 b_1 b_2 c)| &\ll \eps \\
 \E_{c \in (\Z/q\Z)^\times} |f_{1,0}(a_0 c) \psi(b_1 b_2) - f_{1,0}(a_0 b_1 b_2 c)| &\ll \eps.
\end{align*}
Applying the triangle inequality (after multiplying the first inequality by $|\psi(b_2)|$, we conclude that
$$ \E_{c \in (\Z/q\Z)^\times} |f_{1,0}(a_0 c) (\psi(b_1 b_2) - \psi(b_1) \psi(b_2))| \ll \eps$$
and hence by \eqref{f10-large} we have the \emph{quasimorphism equation}
$$ \psi(b_1 b_2) = \psi(b_1) \psi(b_2) + O(\eps).$$

We now apply a stability theorem to replace this quasimorphism on $(\Z/q\Z)^\times$ by a homomorphism (i.e., a Dirichlet character).

\begin{lemma}[Stability of Dirichlet characters]\label{dirs}  Let $\varepsilon>0$, and let $\psi: (\Z/q\Z)^\times \to \C$ be a function obeying the bound $\psi(b) = O(1)$ for all $b \in \Z/q\Z$, the identity $\psi(1) = 1$, and the quasimorphism equation $\psi(b_1 b_2) = \psi(b_1) \psi(b_2) + O(\eps)$ for all $b_1,b_2 \in (\Z/q\Z)^\times$.  Then there exists a Dirichlet character $\chi: (\Z/q\Z)^\times \to \mathbb{S}^{1}$ of period $q$ such that $\psi(b) = \chi(b) + O(\eps)$ for all $b \in (\Z/q\Z)^\times$.
\end{lemma}

This lemma is a special case of Kazhdan\footnote{We thank Assaf Naor for this reference.  Ben Green also pointed out to us the closely related fact that the bounded cohomology of amenable groups is trivial; see for instance \cite[Theorem 3.7]{frigerio}.} \cite{kazhdan}, and also follows from \cite[Proposition 5.3]{bgs} (which cites \cite{bfl} for a more general result), but for the convenience of the reader we give a self-contained proof here.

\begin{proof}  We can assume that $\eps$ is smaller than any given positive absolute constant, as the claim is trivial otherwise.  Since $1 = \psi(1) = \psi(b) \psi(b^{-1}) + O(\eps)$ and $\psi(b^{-1}) = O(1)$, we see that $1 \ll |\psi(b)| \ll 1$ for all $b \in (\Z/q\Z)^\times$.  We introduce the cocycle $\rho: (\Z/q\Z)^\times \times (\Z/q\Z)^\times \to \C$ by defining $\rho(b_1,b_2)$ for $b_1,b_2 \in (\Z/q\Z)^\times$ to be the unique complex number of size $O(\eps)$ such that
\begin{equation}\label{psib}
 \psi(b_1 b_2) = \psi(b_1) \psi(b_2) \exp( \rho(b_1, b_2) );
\end{equation}
this is well-defined for $\eps$ small enough.  For $b_1,b_2,b_3 \in (\Z/q\Z)^\times$, we have
$$ \psi(b_1 b_2 b_3) = \psi(b_1 b_2) \psi(b_3) \exp( \rho(b_1b_2, b_3) ) = \psi(b_1) \psi(b_2) \psi(b_3) \exp( \rho(b_1, b_2) + \rho(b_1 b_2, b_3) )$$
and
$$ \psi(b_1 b_2 b_3) = \psi(b_1) \psi(b_2 b_3) \exp( \rho(b_2, b_3) ) = \psi(b_1) \psi(b_2) \psi(b_3) \exp( \rho(b_1, b_2 b_3) + \rho(b_2, b_3) )$$
which on taking logarithms yields (for $\eps$ small enough) the \emph{cocycle equation}
$$ \rho(b_1, b_2) + \rho(b_1 b_2, b_3) = \rho(b_1, b_2 b_3) + \rho(b_2, b_3).$$
Averaging in $b_3$, we conclude the \emph{coboundary equation}
$$ \rho(b_1, b_2) + \phi(b_1 b_2) = \phi(b_1) + \phi(b_2)$$
where $\phi(b) \coloneqq \E_{b_3 \in (\Z/q\Z)^\times} \rho(b,b_3)$.  If we then define the function $\chi: \Z/q\Z \to \C$ by 
$$ \chi(b) \coloneqq \psi(b) \exp( \phi(b) ),$$
then $\psi(b) = \chi(b) + O(\eps)$ for all $b \in (\Z/q\Z)^\times$, and from \eqref{psib} we have
$$ \chi(b_1 b_2) = \chi(b_1) \chi(b_2)$$
for all $b_1,b_2 \in (\Z/q\Z)^\times$, thus $\chi: (\Z/q\Z)^\times \to \C$ is a homomorphism and therefore a Dirichlet character of period $q$.  The claim follows.
\end{proof}

Let $\chi$ be the Dirichlet character of period $q$ provided by the above lemma, then from \eqref{ecc} and the triangle inequality we have the \emph{approximate isotopy equation}
$$
 \E_{c \in (\Z/q\Z)^\times} |f_{1,0}(a c) \chi(b) - f_{1,0}(a bc)| \ll \eps
$$
for all $a \in \Z/q\Z$ and $b \in (\Z/q\Z)^\times$.  We can rearrange this as
$$ \E_{c \in (\Z/q\Z)^\times} |f_{1,0}(a c) - \overline{\chi(b)} f_{1,0}(a bc)| \ll \eps $$
and average in $b$ to conclude that
\begin{equation}\label{tfa}
 \E_{c \in (\Z/q\Z)^\times} |f_{1,0}(a c) - \tilde f(ac)| \ll \eps 
\end{equation}
for all $a$, where $\tilde f: \Z/q\Z \to \C$ is the function
$$ \tilde f(a) \coloneqq \E_{b \in (\Z/q\Z)^\times} \overline{\chi(b)} f_{1,0}(a b).$$
Observe that $\tilde f$ is $\chi$-isotypic in the sense that
$$ \tilde f(ab) = \chi(b) \tilde f(a)$$
whenever $a \in \Z/q\Z$ and $b \in (\Z/q\Z)^\times$.  

From \eqref{tfa} and \eqref{f10-large}, one has
$$ \E_{c \in (\Z/q\Z)^\times} |\tilde f(a_0 c)| \gg 1$$
and hence by the $\chi$-isotypy of $\tilde f$
\begin{equation}\label{tfa0}
|\tilde f(a_0)| \gg 1.
\end{equation}

Now we work to control $f_d$.  Let $a$ be an integer.  From \eqref{tfa} and the prime number theorem in arithmetic progressions, we have
$$
\limsup_{P_1 \to \infty} \limsup_{P_2 \to \infty} \E_{p_1 \leq P_1}^{\log}  \E_{p_2 \leq P_2}^{\log} |f_{1,0}(a p_1 p_2) - \tilde f(a p_1 p_2)| \ll \eps
$$
From this, \eqref{gp-inv}, \eqref{2m1-0}, and the triangle inequality, we conclude that 
$$
\limsup_{P_1 \to \infty} \limsup_{P_2 \to \infty} \E_{p_1 \leq P_1}^{\log} \E_{p_2 \leq P_2}^{\log} |f_{p_1 p_2}(a) - \overline{G(p_1)}\,  \overline{G(p_2)} \tilde f(a p_1 p_2)| \ll \eps.
$$
Using the $\chi$-isotopy of $\tilde f$, we can write this as
\begin{equation}\label{pip-0}
\limsup_{P_1 \to \infty} \limsup_{P_2 \to \infty} \E_{p_1 \leq P_1}^{\log} \E_{p_2 \leq P_2}^{\log} |f_{p_1 p_2}(a) - \overline{G}\chi(p_1) \overline{G}\chi(p_2) \tilde f(a)| \ll \eps.
\end{equation}
This has the following useful consequence.

\begin{lemma}[Isotopy]\label{isotop}  Let the notation be as above. Let $a$ be an integer and let $b$ be an integer coprime to $q$.  Then we have
$$ \limsup_{X\to \infty}\E_{d \leq X}^{\log\log} |f_d(ab) - \chi(b) f_d(a)| \ll \eps.$$
\end{lemma}

\begin{proof} It suffices to prove the claim with an arbitrary generalised limit $\lim^{*}_{X\to \infty}$ in place of $\limsup_{X\to \infty}$. From \eqref{pip-0} we have
$$
 \E_{p_1 \in {\P}}^{\log,*} \E_{p_2 \in {\P}}^{\log,*} |f_{p_1 p_2}(a) - \overline{G}\chi(p_1) \overline{G}\chi(p_2) \tilde f(a)| \ll \eps$$
and
$$
 \E_{p_1 \in {\P}}^{\log,*} \E_{p_2 \in {\P}}^{\log,*} |f_{p_1 p_2}(ab) - \overline{G}\chi(p_1) \overline{G}\chi(p_2) \tilde f(ab)| \ll \eps.$$
As $\tilde f$ is isotypic, $\tilde f(ab) = \chi(b) \tilde f(a)$.  From the triangle inequality and \eqref{gp-inv}, we conclude that
$$
 \E_{p_1 \in {\P}}^{\log,*} \E_{p_2 \in {\P}}^{\log,*} |f_{p_1 p_2}(ab) - \chi(b) f_{p_1 p_2}(a)| \ll \eps.$$
On the other hand, since $d\mapsto |f_{d}(ab)-\chi(b)f_{d}(a)|$ is bounded log-Lipschitz by \eqref{fd1d2}, by Lemma \ref{chug} for any $p_1$ we have
$$
\E_{p_2 \in {\P}}^{\log,*} |f_{p_1 p_2}(ab) - \chi(b) f_{p_1 p_2}(a)|=\E_{d\in \N}^{\log\log, *}|f_{d}(ab)-\chi(b)f_{d}(a)|,
$$
and now the claim now follows by taking the average $\E_{p_1\in \P}^{\log,*}$ on both sides.
\end{proof}

Now we derive another consequence of \eqref{pip-0}. Let $x>0$ be a positive real, and let $a$ be an integer.  From \eqref{pip-0} we have
$$
 \E_{p'_1 \in {\P}}^{\log, **} \E_{p_2 \in {\P}}^{\log,**} |f_{p'_1 p_2}(a) - \overline{G}\chi(p'_1) \overline{G}\chi(p_2) \tilde f(a)| \ll \eps.
$$
By the prime number theorem, this can also be written as
$$
 \E_{p_1 \in {\P}}^{\log,**} \E_{x p_1 \leq p'_1 \leq (1+\eps) xp_1} \E_{p_2 \in {\P}}^{\log,**} |f_{p'_1 p_2}(a) - \overline{G}\chi(p'_1) \overline{G}\chi(p_2) \tilde f(a)| \ll \eps.
$$
From \eqref{gp-inv} we have
$$1-\varepsilon \leq |\overline{G}\chi(p_1)|, |\overline{G}\chi(p'_1)| , |\overline{G}\chi(p_2)|\leq 1$$
for all but finitely many $p_1,p_1',p_2$, so that
\begin{align*} 
&|f_{p'_1 p_2}(a) - \overline{G}\chi(p'_1) G \overline{\chi}(p_1) f_{p_1 p_2}(a)|\\
&\ll
|f_{p_1 p_2}(a) - \overline{G}\chi(p_1) \overline{G}\chi(p_2) \tilde f(a)| + |f_{p'_1 p_2}(a) - \overline{G}\chi(p'_1) \overline{G}\chi(p_2) \tilde f(a)|+O(\varepsilon).
\end{align*}
Thus by the triangle inequality we have 
$$
 \E_{p_1 \in {\P}}^{\log,**} \E_{x p_1 \leq p'_1 \leq (1+\eps) xp_1} \E_{p_2 \in {\P}}^{\log,**}  |f_{p'_1 p_2}(a) - \overline{G}\chi(p'_1) G \overline{\chi}(p_1) f_{p_1 p_2}(a)| \ll \eps.
$$
From \eqref{fd1d2} we have $f_{p'_1 p_2}(a) = f_{x p_1 p_2}(a) + O(\eps)$ (recall that $f_d$ is defined for any real $d>0$), thus
$$
 \E_{p_1 \in {\P}}^{\log,**} \E_{x p_1 \leq p'_1 \leq (1+\eps) xp_1} \E_{p_2 \in {\P}}^{\log,**} |f_{x p_1 p_2}(a) - \overline{G}\chi(p'_1) G \overline{\chi}(p_1) f_{p_1 p_2}(a)| \ll \eps
$$
and thus by the triangle inequality
$$ \E_{p_1 \in {\P}}^{\log,**} \E_{p_2 \in {\P}}^{\log,**} |f_{x p_1 p_2}(a) - \alpha_{p_1}(x) f_{p_1 p_2}(a)| \ll \eps,
$$
where
$$ \alpha_{p_1}(x) \coloneqq \E_{x p_1 \leq p'_1 < (1+\eps) x p_1} \overline{G}\chi(p'_1) G \overline{\chi}(p_1).$$
By Lemma \ref{chug}, this implies that
\begin{equation}\label{edl-0}
 \E_{p_1 \in {\P}}^{\log,**} \E_{d \in \N}^{\log\log,**} |f_{x d}(a) - \alpha_{p_1}(x) f_{d}(a)| \ll \eps
\end{equation}
which by the triangle inequality implies that
\begin{equation}\label{edl}
 \E_{d \in \N}^{\log\log,**} |f_{x d}(a) - \alpha(x) f_{d}(a)| \ll \eps
\end{equation}
where
$$ \alpha(x) \coloneqq \E_{p_1 \in {\P}}^{\log,**} \alpha_{p_1}(x).$$

By construction, we have $\alpha(x) = O(1)$ for all $x$. Setting $a=a_0$ in \eqref{edl} and using \eqref{fd1d2} to write $f_{xd}(a)=f_{d}(a)+O(\varepsilon)$ for $|x-1|\leq \varepsilon$, we deduce from \eqref{fda} that $\alpha(x)=1+O(\varepsilon)$ for $|x-1|\leq \varepsilon$.

Next, for $x,y > 0$, we have the estimates
\begin{align*}
\E_{d \in \mathbb{N}}^{\log\log, **} |f_{x d}(a_0) - \alpha(x) f_{d}(a_0)| &\ll \eps\\
\E_{d \in \mathbb{N}}^{\log\log, **} |f_{x y d}(a_0) - \alpha(y) f_{xd}(a_0)| &\ll \eps\\
\E_{d \in \mathbb{N}}^{\log\log, **} |f_{x y d}(a_0) - \alpha(xy) f_{d}(a_0)| &\ll \eps,
\end{align*}
which by the triangle inequality and \eqref{fda} implies the quasimorphism equation
$$ \alpha(xy) = \alpha(x) \alpha(y) + O(\eps).$$
We now require the Archimedean analogue of Lemma \ref{dirs} (which is also a special case of the results of \cite{kazhdan}).

\begin{lemma}[Stability of Archimedean characters]\label{dira}  Let $\alpha: (0,+\infty) \to \C$ be any function obeying the bound $\alpha(x) = O(1)$ for all $x > 0$, such that $\alpha(x) = 1 + O(\eps)$ when $|x-1| \leq \eps$, and $\alpha(xy) = \alpha(x) \alpha(y) + O(\eps)$ for all $x,y>0$.  Then there exists a real number $t$ such that $\alpha(x) = x^{-it} + O(\eps)$ for all $x>0$.
\end{lemma}

\begin{proof} As before, we can assume $\eps$ is smaller than any given positive constant, as the claim is trivial otherwise.  Since $\alpha(1) = 1+O(\eps)$ and $\alpha(1) = \alpha(x) \alpha(1/x) + O(\eps)$, we have the bounds $1 \ll |\alpha(x)| \ll 1$ for all $x$.  By construction, we also have $\alpha(xy) = \alpha(x) + O(\eps)$ whenever $1 \leq y \leq 1+\eps$.  By replacing $\alpha$ with the discretised version 
\begin{align*}
\alpha_1(x):=\begin{cases}\alpha(\varepsilon^2\lfloor \frac{x}{\varepsilon^2}\rfloor),\quad x\geq \varepsilon,\\ \alpha(\frac{1}{n}),\quad x\in (\frac{1}{n+1},\frac{1}{n}],\,\, 0<x<\varepsilon,\end{cases}
\end{align*}
we may assume that $\alpha$ is Lebesgue measurable. The function $\alpha_1$ continues to enjoy the same properties as $\alpha$, since $\alpha_1(x)=\alpha(x)+O(\varepsilon)$ for all $x>0$. To simplify notation, we denote $\alpha_1$ by $\alpha$ in what follows.

We introduce the cocycle $\rho: (0,+\infty) \times (0,+\infty) \to \C$ by defining $\rho(x_1,x_2)$ for $x_1,x_2 > 0$ to be the unique complex number of size $O(\eps)$ such that
\begin{align}\label{eq3}
 \alpha(x_1 x_2) = \alpha(x_1) \alpha(x_2) \exp( \rho(x_1, x_2) );
\end{align}
this is well-defined and measurable for $\eps$ small enough.  Arguing exactly as in the proof of Lemma \ref{dirs}, we obtain the cocycle equation
$$ \rho(x_1, x_2) + \rho(x_1 x_2, x_3) = \rho(x_1, x_2 x_3) + \rho(x_2, x_3).$$
Taking an asymptotic logarithmic average in $x_3$, we conclude the coboundary equation
\begin{equation}\label{eq21}
\rho(x_1, x_2) + \phi(x_1 x_2) = \phi(x_1) + \phi(x_2)
\end{equation}
where
$$ \phi(x) \coloneqq \plim_{M \to \infty} \frac{1}{\log M} \int_{1}^{M} \rho(x, x_3) \frac{dx_3}{x_3}.$$
If we then define the function $\tilde \alpha: (0,+\infty) \to \C$ by
$$ \tilde \alpha(x) \coloneqq \alpha(x) \exp( \phi(x) )$$
then $\tilde \alpha(x) = \alpha(x) + O(\eps)$ for all $x>0$, and from \eqref{eq3} and \eqref{eq21} we have
$$ \tilde \alpha(xy) = \tilde \alpha(x) \tilde \alpha(y)$$
for all $x,y > 0$, thus $\tilde \alpha: (0,+\infty) \to \C$ is a homomorphism.  Also, by construction one has $\tilde \alpha(x) = O(1)$ for all $x$, so $\tilde \alpha$ in fact takes values in the unit circle $\mathbb{S}^{1}$.  We have $\tilde \alpha(x) = 1+O(\eps)$ when $|x-1| \leq \eps$, and we will use this additional information to show that $\tilde \alpha(x)=x^{it}$ for some real $t$ and all $x>0$. 

If $|x-1| \leq \eps/n$ for some natural number $n$, then $\tilde \alpha(x)^n, \tilde \alpha(x) = 1 + O(\eps)$, which implies that $\tilde \alpha(x) = 1 + O(\eps/n)$.  This implies that $\tilde \alpha(x) = 1 + O(|x-1|)$, and so $\tilde \alpha$ is continuous at $1$ and hence continuous on all of $(0,+\infty)$.  Next, if $x_0 \coloneqq 1 + \eps$ then we have $\tilde \alpha(x_0) = x_0^{it}$ for some $t=O(1)$; taking roots we conclude that $\tilde \alpha(x_0^{1/n}) = (x_0^{1/n})^{it}$ for all natural numbers $n$, and hence $\tilde \alpha(x_0^{m/n}) = (x_0^{m/n})^{it}$ for all natural numbers $n$ and integers $m$.  By continuity we conclude that $\tilde \alpha(x) = x^{it}$ for all $x \in (0,+\infty)$, as required.
\end{proof}

From the above lemma, we conclude that there is a real number $t$ with the property that for every integer $a$ and real $x>0$, one has
\begin{equation} \label{eq25} \E_{d \in \N}^{\log\log,**} |f_{x d}(a) - x^{-it} f_{d}(a)| \ll \eps.\end{equation}
In particular, for every prime $p_1$, one has
\begin{equation*} \E_{d \in \N}^{\log\log,**} |f_{p_1 d}(a_0) - p_1^{-it} f_{d}(a_0)| \ll \eps,\end{equation*}
and thus
\begin{equation}\label{eq22} \E_{p_1 \in {\P}}^{\log,**} \E_{d \in \N}^{\log\log,**} |f_{p_1 d}(a_0) - p_1^{-it} f_{d}(a_0)| \ll \eps\end{equation}

On the other hand, from Proposition \ref{dollop} one has that if $P_1$ is sufficiently large depending on $a_0,\eps$, then
$$ \sup_{d>0} \E^{\log}_{p_1 \leq P_1} |f_{p_1 d}(a_0) G(p_1) - f_d(a_0p_1)| \ll \eps.$$
Hence on averaging in $d$ and taking limits in the $d$ average and then in the $p_1$ average, we conclude that
\begin{equation}\label{eq23} \limsup_{P_1 \to \infty} \E^{\log}_{p_1 \leq P_1} \E_{d \in \N}^{\log\log,**} |f_{p_1 d}(a_0) G(p_1) - f_d(a_0p_1)| \ll \eps.\end{equation}
Meanwhile, from Lemma \ref{isotop} we have
$$ \E_{d \in \N}^{\log\log,**} |f_d(a_0p_1) - \chi(p_1) f_d(a_0)| \ll \eps$$
for all sufficiently large $p_1$, and thus
\begin{equation} \label{eq24} \limsup_{P_1 \to \infty} \E^{\log}_{p_1 \leq P_1} \E_{d \in \N}^{\log\log,**} |f_d(a_0p_1) - \chi(p_1) f_d(a_0)| \ll \eps.\end{equation}
Applying the triangle inequality to \eqref{eq22}, \eqref{eq23}, \eqref{eq24}, we obtain
$$ \limsup_{P_1 \to \infty} \E^{\log}_{p_1 \leq P_1} \E_{d \in \N}^{\log\log,**} |G(p_1) - \chi(p_1) p_1^{it}| |f_d(a_0)| \ll \eps $$
and hence by \eqref{fda} we have
$$ \limsup_{P_1 \to \infty} \E^{\log}_{p_1 \leq P_1} |G(p_1) - \chi(p_1) p_1^{it}| \ll \eps.$$

To summarise the above analysis, we have shown that for every $\eps>0$ there exists a Dirichlet character $\chi = \chi_\eps$ and a real number $t = t_\eps$ such that
$$ \limsup_{P_1 \to \infty} \E^{\log}_{p_1 \leq P_1} |G(p_1) - \chi_\eps(p_1) p_1^{it_\eps}| \ll \eps.$$
\emph{A priori}, the character $\chi_\eps$ and the real number $t_\eps$ depend on $\eps$.  But if $\eps,\eps'>0$ are sufficiently small, we have from the triangle inequality that
$$ \limsup_{P_1 \to \infty} \E^{\log}_{p_1 \leq P_1} |\chi_{\eps'}(p_1) p_1^{it_{\eps'}} - \chi_\eps(p_1) p_1^{it_\eps}| \ll \eps + \eps'.$$
But from the prime number theorem in arithmetic progressions and partial summation, we see that the left-hand side is $\gg 1$ unless $t_\eps = t_{\eps'}$ and the Dirichlet characters are \emph{co-trained} in the sense that they are both induced from the same primitive character $\chi$.  We conclude that there exists a primitive character $\chi$ independent of $\eps$, and a real number $t_0$ independent of $\eps$, such that $t_\eps = t_0$ and $\chi_\eps$ is induced from $\chi$ for $\eps$ sufficiently small.  In particular, as $\chi_{\varepsilon}(p_1)$ and $\chi(p_1)$ agree for all but $O_{\varepsilon}(1)$ primes $p_1$, we have for each $\eps>0$ that
$$
\limsup_{P_1 \to \infty} \E^{\log}_{p_1 \leq P_1} |G(p_1) - \chi(p_1) p_1^{it_0}| \ll \eps
$$
and thus
\begin{equation}\label{lop}
\E^{\log}_{p_1 \in {\P}} |G(p_1) - \chi(p_1) p_1^{it_0}| = 0.
\end{equation}
 Thus $G$ weakly pretends to be the twisted Dirichlet character $n \mapsto n^{it_0} \chi(n)$. This (vacuously) establishes part (i) of Theorem \ref{main}.  

Now let $\eps>0$ be small, and let $a$ be an integer.  From \eqref{pip-0} (and the fact that $\chi_\eps$ is induced from $\chi$), and making the dependence of $\tilde f_\eps$ on $\eps$ explicit, we have
$$
 \E_{p_1 \in {\P}}^{\log,**} \E_{p_2 \in {\P}}^{\log,**} |f_{p_1 p_2}(a) - \overline{G}\chi(p_1) \overline{G}\chi(p_2) \tilde f_\eps(a)| \ll \eps
$$
and hence by \eqref{lop} and the triangle inequality
$$
 \E_{p_1 \in {\P}}^{\log,**} \E_{p_2 \in {\P}}^{\log,**} |f_{p_1 p_2}(a) - p_1^{-it_0} p_2^{-it_0} \tilde f_\eps(a)| \ll \eps
$$
or equivalently
$$
 \E_{p_1 \in {\P}}^{\log,**} \E_{p_2 \in {\P}}^{\log,**} |(p_1p_2)^{it_0} f_{p_1 p_2}(a) - \tilde f_\eps(a)| \ll \eps.
$$
Applying \eqref{fd1d2}, Lemma \ref{chug} and \eqref{eq25} (where we can in fact take $\varepsilon\to 0$, since the deduction succeeding this formula shows that $t=t_0$ is independent of $\varepsilon$), we have
$$\E_{p_2 \in {\P}}^{\log,**} |(p_1p_2)^{it_0} f_{p_1 p_2}(a) - \tilde f_\eps(a)| = 
\E_{d \in \N}^{\log\log,**} |(p_1d)^{it_0} f_{p_1 d}(a) - \tilde f_\eps(a)|  = \E_{d \in \N}^{\log\log,**} |d^{it_0} f_d(a) - \tilde f_\eps(a)| 
$$
for any $p_1$, and hence
$$
 \E_{d \in \N}^{\log\log,**} |d^{it_0} f_{d}(a) - \tilde f_\eps(a)| \ll \eps.
$$
We thus see from the triangle inequality that
$$ |\tilde f_\eps(a) - \tilde f_{\eps'}(a)| \ll \eps + \eps'$$
and so $\tilde f_\eps$ converges uniformly to a limit $f$ with
\begin{equation}\label{ttfa}
 |\tilde f_\eps(a) - f(a)| \ll \eps
\end{equation}
and thus by the triangle inequality, we have
$$
 \E_{d \in \N}^{\log\log,**} |d^{it_0} f_{d}(a) - f(a)| \ll \eps
$$
whenever $\eps>0$, which gives
\begin{equation}\label{dlog}
 \E_{d \in \N}^{\log\log,**} |d^{it_0} f_{d}(a) - f(a)| = 0.
\end{equation}
From \eqref{fda} we see in particular that $f(a_0) \neq 0$.  By construction, each $\tilde f_\eps$ is $\chi$-isotypic in the sense that $\tilde f_\eps(ab) = \chi(b) \tilde f_\eps(a)$ whenever $a,b$ are integers with $b$ coprime to the periods of both $\chi$ and $\tilde f_\eps$. Hence, what remains to be shown is that \eqref{dlog} holds also when taking the average with respect to the ordinary limit.

Now let $\eps>0$ be arbitrary.  Inserting \eqref{lop} into \eqref{pip-0}, we see that
$$
\limsup_{P_1 \to \infty} \limsup_{P_2 \to \infty} \E_{p_1 \leq P_1}^{\log} \E_{p_2 \leq P_2}^{\log} |f_{p_1 p_2}(a) - (p_1 p_2)^{-it_0} \tilde f_\eps(a)| \ll \eps
$$
and hence by \eqref{ttfa} and sending $\eps \to 0$ we get
\begin{equation*}
\limsup_{P_1 \to \infty} \limsup_{P_2 \to \infty} \E_{p_1 \leq P_1}^{\log} \E_{p_2 \leq P_2}^{\log} |f_{p_1 p_2}(a) - (p_1 p_2)^{-it_0} f(a)| = 0.
\end{equation*}
For any $\varepsilon>0$ and any $P_1$ large enough in terms of $\varepsilon$, we apply Lemma \ref{chug}, Proposition \ref{dollop}, formula \eqref{lop} and Lemma \ref{isotop} to write
\begin{align*}
 &\limsup_{P_2 \to \infty} \E_{p_1 \leq P_1}^{\log} \E_{p_2 \leq P_2}^{\log} |f_{p_1 p_2}(a) - (p_1 p_2)^{-it_0} f(a)|\\ 
=& \limsup_{P_2 \to \infty} \E_{p_1 \leq P_1}^{\log} \E_{d \leq P_2}^{\log\log} |f_{p_1 d}(a) - (p_1 d)^{-it_0} f(a)|  \\
=& \limsup_{P_2 \to \infty} \E_{p_1 \leq P_1}^{\log} \E_{d \leq P_2}^{\log\log} |\overline{G(p_1)}f_d(ap_1) - (p_1 d)^{-it_0} f(a)|+O(\varepsilon)\\
=&\limsup_{P_2 \to \infty} \E_{p_1 \leq P_1}^{\log} \E_{d \leq P_2}^{\log\log} |p_1^{-it_0}\overline{\chi}(p_1)f_d(ap_1) - (p_1 d)^{-it_0} f(a)|+O(\varepsilon)\\
=& \limsup_{P_2 \to \infty} \E_{d \leq P_2}^{\log\log} |f_{d}(a) - d^{-it_0} f(a)|+O(\varepsilon),  
\end{align*}
and hence, sending $\varepsilon\to 0$, we obtain
$$ \E_{d \in \N}^{\log\log} |f_d(a) - d^{-it_0} f(a)| = 0.$$
This establishes part (ii) of Theorem \ref{main} (recalling as before that as $G$ weakly pretends to be a twisted Dirichlet character $n \mapsto \chi(n) n^{it}$, it can only weakly pretend to be another twisted Dirichlet character $n \mapsto \chi'(n) n^{it'}$ if $t=t'$ and $\chi, \chi'$ are co-trained).

\section{Proofs of corollaries}\label{corollaries-sec}

In this section we use Theorem \ref{main} to prove Corollaries \ref{elliott-1}, \ref{elliott-2}, \ref{chow}.  We begin with Corollary \ref{elliott-1}.  

\begin{proof}[Proof of Corollary \ref{elliott-1}] Suppose the claim failed, then we can find $k, g_1,\dots,g_k$ as in that corollary, as well as $h_1,\dots,h_k \in \Z$ and $\eps>0$, such that the set
$$ {\mathcal X} \coloneqq \{ X \in \N: |\E_{n \leq X} g_1(n+h_1) \cdots g_k(n+h_k)| > \eps \}$$
does not have logarithmic Banach density zero.  In particular, one can find sequences $X_i \geq \omega_i \to \infty$ and $0 < \delta < 1/2$ such that
\begin{align}\label{eq4}
\E^{\log}_{X_i/\omega_i \leq X \leq X_i} 1_{{\mathcal X}}(x) \geq \delta
\end{align} 
for all $i$. 

Intuitively, if the exceptional set $\mathcal{X}$ was big in the sense of \eqref{eq4}, there would have to be a lot of ``points of density'' of $\mathcal{X}$ (in a sense to be specified later). To make this rigorous, we introduce for each $i$ the function $a_i: \R \to [0,1]$ given by
$$ a_i(s) \coloneqq \sum_{X_i/\omega_i \leq X \leq X_i: X \not \in {\mathcal X}} 1_{\log(X-1) < s \leq \log X}.$$
Note that $a_i(s)$ is the indicator function of the event that there exists an integer $X\not \in \mathcal{X}$ with $X\in [e^{s},e^{s}+1)$ and $X_i/\omega_i\leq X\leq X_i$.

The function $a_i$ is a piecewise constant function supported on an interval of length $(1 + o_{i \to \infty}(1)) \log \omega_i$ and has integral
\begin{align*}
\int_\R a_i(s)\ ds &= \sum_{X_i/\omega_i \leq X \leq X_i: X \not \in {\mathcal X}} \log \frac{X}{X-1} \\
&= \left(\sum_{X_i/\omega_i \leq X \leq X_i: X \not \in {\mathcal X}} \frac{1}{X}\right) + O(1) \\
&= \log \omega_i + O(1) - \sum_{X_i/\omega_i \leq X \leq X_i: X \in {\mathcal X}} \frac{1}{X} \\
&\leq (1-\delta + o_{i \to \infty}(1)) \log \omega_i.
\end{align*}
We introduce the one-sided Hardy--Littlewood maximal function
$$ Ma_i(s) \coloneqq \sup_{r>0} \frac{1}{r} \int_{s-r}^s a_i(s')\ ds'.$$
It is a well-known consequence of the rising sun lemma \cite{riesz} that one has the Hardy--Littlewood maximal inequality 
$$ m( \{ s \in \R: Ma_i(s) \geq \lambda \} ) \leq \frac{1}{\lambda} \int_\R a_i(s)\ ds$$
for any $\lambda>0$, where $m$ denotes Lebesgue measure.  Applying this with $\lambda \coloneqq (1-\delta)^{1/2}$, we conclude that
$$ m(\{ s \in \R: Ma_i(s) \geq (1-\delta)^{1/2} \} ) \leq ((1-\delta)^{1/2} + o_{i \to \infty}(1)) \log \omega_i.$$
In particular, one can find a real number $s_i$ with
\begin{equation}\label{txi}
 \log X_i - ((1-\delta)^{1/2} + o_{i \to \infty}(1)) \log \omega_i \leq s_i \leq \log X_i
\end{equation}
such that
$$ Ma_i(s_i) < (1-\delta)^{1/2} $$
which implies that
\begin{equation}\label{ar}
 \int_{s_i-r}^{s_i} a_i(t)\ dt \leq (1-\delta)^{1/2} r
\end{equation}
for all $r>0$.  Informally, the estimate \eqref{ar} asserts that the natural number $\lfloor \exp(s_i) \rfloor$ is a ``multiplicative point of density'' for the exceptional set ${\mathcal X}$.

By passing to subsequences, and using a diagonalisation argument, we may assume that the limits
\begin{align}\label{eq5}
 f_d(a) \coloneqq \lim_{i \to \infty}  \E_{n \leq \lfloor \exp(s_i)\rfloor/d} g_1(n+ah_1) \cdots g_k(n+ah_k),    
\end{align}
exist for every natural number $d$ and integer $a$.  In particular, the limit of the right-hand side of \eqref{eq5} is the same along any generalised limit $\lim^{**}$. If we now apply Theorem \ref{main}(i) to a generalised limit of the form
$$ \lim^*_{X \to \infty} f(X) \coloneqq \lim^{**}_{i \to \infty} f( \lfloor \exp(s_i) \rfloor ),$$
where $\lim^{**}$ is any generalised limit, we conclude that
$$ \E^{\log\log}_{d \in \N} |f_d(1)| = 0.$$
Thus, if we let $\mu>0$ denote a small constant (depending on $\delta,\eps$) to be chosen later, and $D$ is sufficiently large depending on $\mu$, we have
$$ \E^{\log\log}_{d \leq D} |f_d(1)| \leq \mu.$$
Thus by the triangle inequality
$$ \limsup_{i \to \infty} \E^{\log\log}_{d \leq D} |\E_{n \leq \lfloor \exp(s_i)\rfloor/d} g_1(n+h_1) \cdots g_k(n+h_k)| \leq \mu, $$
and hence for all sufficiently large $i$ (depending on $\delta,\eps,\mu,D$) we find
$$ \E_{d\leq D}^{\log \log} |\E_{n \leq \lfloor \exp(s_i)\rfloor/d} g_1(n+h_1) \cdots g_k(n+h_k)| \leq 2\mu. $$
This implies
$$ \sum_{\log D \leq d \leq D} \frac{1}{d \log d} |\E_{n \leq \lfloor \exp(s_i)\rfloor/d} g_1(n+h_1) \cdots g_k(n+h_k)| \ll \mu \log\log D, $$
say.  In particular, by Markov's inequality one has
\begin{equation}\label{df}
|\E_{n \leq \lfloor \exp(s_i)\rfloor/d} g_1(n+h_1) \cdots g_k(n+h_k)| \leq \frac{\eps}{2} 
\end{equation}
for all $\log D \leq d \leq D$ outside of an exceptional set ${\mathcal D}_i$ with
\begin{equation}\label{di}
 \sum_{d\in {\mathcal D}_i} \frac{1}{d \log d} \ll \frac{\mu}{\eps} \log\log D.
\end{equation}
If $\log D \leq d \leq D$ lies outside of ${\mathcal D}_i$, then one has
$$  |\E_{n \leq X} g_1(n+h_1) \cdots g_k(n+h_k)| < \eps$$
for all $X$ between $\frac{\exp(s_i)}{d+1}-1$ and $\frac{\exp(s_i)}{d}+1$.  In particular, all such $X$ lie outside of ${\mathcal X}$.  Using \eqref{txi} (which places $\exp(s_i)/d$ below $X_i$ and well above $X_i/\omega_i$), we conclude that
$$ a_i(t) = 1$$
on the interval $[s_i - \log(d+1), s_i - \log(d)]$.  In particular,
$$ \int_d^{d+1} a_i( s_i - \log u)\ du = 1.$$
For $d \in {\mathcal D}_i$ we use the trivial bound
$$ \int_d^{d+1} a_i( s_i - \log u)\ du \geq 0.$$
From \eqref{di} we conclude that
\begin{equation}\label{eq27} 
\sum_{\log D \leq d \leq D} \frac{1}{d \log d} \int_d^{d+1} a_i( s_i - \log u)\ du  \geq (1 - O(\frac{\mu}{\eps})) \log\log D. \end{equation}
The left-hand side, up to errors that can be absorbed into the $O( \frac{\mu}{\eps}) \log\log D$ term, can be rewritten as
$$ \int_{\log D}^D a_i( s_i - \log u) \frac{du}{u\log u}$$
which by the change of variables $s = s_i - \log u$ becomes
$$ \int_{s_i - \log D}^{s_i - \log\log D} a_i(s) \frac{ds}{s_i-s}.$$
However, from Fubini's theorem and \eqref{ar} we have
\begin{align*}
\int_{s_i - \log D}^{s_i - \log\log D} a_i(s) \frac{ds}{s_i-s} 
&= \int_{s_i - \log D}^{s_i - \log\log D} a_i(s) (\int_{s_i - \log D}^s \frac{dt}{(s_i-t)^2} + \frac{1}{\log D})\ ds \\
&= \int_{s_i-\log D}^{s_i - \log\log D} (\int_t^{s_i-\log\log D} a_i(s)\ ds) \frac{dt}{(s_i-t)^2} + \frac{1}{\log D} \int_{s_i-\log D}^{s_i-\log\log D} a_i(s)\ ds \\
&\leq
 \int_{s_i-\log D}^{s_i - \log\log D} (\int_t^{s_i} a_i(s)\ ds) \frac{dt}{(s_i-t)^2} + \frac{1}{\log D} \int_{s_i-\log D}^{s_i} a_i(s)\ ds \\
&\leq 
 \int_{s_i-\log D}^{s_i - \log\log D} (1 - \delta)^{1/2} (s_i-t) \frac{dt}{(s_i-t)^2} + \frac{1}{\log D} (1-\delta)^{1/2} \log D \\
&= (1-\delta)^{1/2} (\log\log D - \log\log\log D + 1)
\end{align*}
and the right-hand side is equal to $(1-\delta)^{1/2} \log\log D$ up to errors that can be absorbed into the $O( \frac{\mu}{\eps}) \log\log D$ term.  For $\mu$ small enough, this gives a contradiction when compared with \eqref{eq27}, proving Corollary \ref{elliott-1}(i).

We are left with proving part (ii) of Corollary \ref{elliott-1}. Since sets of logarithmic Banach density zero automatically have logarithmic density zero, we already know from Corollary \ref{elliott-1}(i) that for each tuple $(h_1,\dots,h_k)$ of integers and every $m \geq 1$, there is a set ${\mathcal X}_{h_1,\dots,h_k,m}$ of logarithmic density zero such that
$$
| \E_{n \leq X} g_1(n+h_1) \cdots g_k(n+h_k) |\leq \frac{1}{m}$$
for all $X$ outside of ${\mathcal X}_{h_1,\dots,h_k,m}$.  Since the number of tuples $(h_1,\dots,h_k,m)$ is countable, a standard diagonalisation construction then gives a further set ${\mathcal X}_0$, still of logarithmic density zero, such that for each $h_1,\dots,h_k,m$, all but finitely many of the elements of ${\mathcal X}_{h_1,\dots,h_k,m}$ are contained in ${\mathcal X}_0$.  For instance, one could remove finitely many elements from ${\mathcal X}_{h_1,\dots,h_k,m}$ to create a subset ${\mathcal X}'_{h_1,\dots,h_k,m}$ with the property that
$$ \E_{X \leq Y}^{\log} 1_{{\mathcal X}'_{h_1,\dots,h_k,m}}(X) \leq 2^{-h_1-\cdots-h_k-m}$$
for all $Y \geq 1$, and then take ${\mathcal X}_0$ to be the union of all the ${\mathcal X}'_{h_1,\dots,h_k,m}$, which thus differs from a finite union of these sets by a set of arbitrarily small logarithmic density (and finite unions of the sets ${\mathcal X}'_{h_1,\dots,h_k,m}$ have logarithmic density $0$).  By construction one then has
$$ \limsup_{X \to \infty; X \not \in {\mathcal X}_0} | \E_{n \leq X} g_1(n+h_1) \cdots g_k(n+h_k) |\leq \frac{1}{m}$$
for all $h_1,\dots,h_k,m$, and the claim follows.
\end{proof}

\begin{remark}\label{18-rem} An inspection of the above argument shows that one could have replaced the sequence $n \mapsto g_1(n+h_1) \dots g_1(n+h_k)$ by any other bounded sequence $n \mapsto F(n)$ for which the analogue of Theorem \ref{main}(i) holds, or more precisely that 
$$ \E^{\log\log}_{d \in \N} |\lim^*_{X \to \infty} \E_{n \leq X/d} F(n)| = 0$$
for any generalised limit $\lim^*_{X \to \infty}$.
\end{remark}

Next we prove Corollary \ref{elliott-2}.

\begin{proof}[Proof of Corollary \ref{elliott-2}] By Corollary \ref{elliott-1}, we are done unless $g_1 g_2$ weakly pretends to be a twisted Dirichlet
character $n \mapsto \chi(n)n^{it}$, so suppose that this is indeed the case for some $\chi$ and $t$.   Then for any generalised limit $\lim^*_{X \to \infty}$, the corresponding correlations $f_d(a)$ defined by \eqref{fda-def} obey the property \eqref{hd} for some function $f: \Z \to \Disk$.  If this function $f$ was vanishing at $a=1$ for every choice of the generalised limit, then one could repeat the proof of Corollary \ref{elliott-1} to obtain the claim (cf. Remark \ref{18-rem}).  Thus suppose instead that we can find a generalised limit $\lim^*_{X \to \infty}$ such that $f(1) \neq 0$ for the function $f$ provided by Theorem \ref{main}(ii).  By \eqref{hd} and the triangle inequality, this implies that
$$ \E^{\log\log}_{d \in \N} f_d(1) d^{it} = f(1) \neq 0.$$
In particular, for $D$ sufficiently large, one has
$$ |\E_{d \leq D}^{\log\log}  f_d(1) d^{it} | \gg 1$$
and hence by summation by parts we have
$$ |\E_{d \leq D}^{\log}  f_d(1) d^{it} | \gg 1$$
for a sequence of arbitrarily large $D$.   If $D$ obeys the above estimate, then by \eqref{fda-def} we have
$$ |\lim^*_{X \to \infty} \E_{d \leq D}^{\log} d^{it} \E_{n \leq X/d} g_1(n+h_1) g_2(n+h_2)| \gg 1$$
and thus there exist arbitrarily large $X$ such that
$$ |\E_{d \leq D}^{\log} d^{it} \E_{n \leq X/d} g_1(n+h_1) g_2(n+h_2)| \gg 1.$$
This implies that
$$ |\E_{d \leq D}^{\log} d^{it} \E_{cX/d \leq n \leq X/d} g_1(n+h_1) g_2(n+h_2)| \gg 1$$
for some small constant $c>0$ (not depending on $D$ and $X$).  This yields
$$ \bigg|\sum_{\log D\leq d \leq D} d^{it} \sum_{cX/d \leq n \leq X/d} g_1(n+h_1) g_2(n+h_2)\bigg| \gg X \log D$$
The left-hand side can be rearranged (discarding negligible errors, assuming $D$ is large enough) as
$$ \bigg|\sum_{cX/D \leq n \leq X/\log D} \bigg(\sum_{cX/n \leq d \leq X/n} d^{it}\bigg) g_1(n+h_1) g_2(n+h_2)\bigg| \gg X \log D.$$
By summation by parts, for $cX/D\leq n \leq X/\log D$ we have
\begin{align*}
\sum_{cX/n \leq d \leq X/n} d^{it} = \alpha n^{-it} \frac{X}{n} + o_{D \to \infty}(1),\quad \alpha=\frac{X^{it}-(cX)^{it}\cdot c}{1+it},
\end{align*}
where in particular the quantity $\alpha$ is bounded and is independent of $n$.  For $D$ large enough, we conclude that
$$ \bigg|\sum_{cX/D \leq n \leq X/\log D} \frac{n^{-it}}{n} g_1(n+h_1) g_2(n+h_2)\bigg| \gg \log D.$$
and hence
$$ |\E_{X/D \leq n \leq X}^{\log} n^{-it} g_1(n+h_1) g_2(n+h_2)| \gg 1.$$
Approximating $n^{-it}$ by $(n+h_1)^{-it}$, we conclude that there exist arbitrarily large $D$ such that
$$ |\E_{X/D \leq n \leq X}^{\log} (n+h_1)^{-it} g_1(n+h_1) g_2(n+h_2)| \gg 1$$
for arbitrarily large $X$.  But this contradicts the $k=2$ case of the logarithmically averaged Elliott conjecture \cite[Corollary 1.5]{tao} applied to the functions $n \mapsto n^{-it} g_1(n)$ and $n \mapsto g_2(n)$ (note that the hypothesis \eqref{suptx} for $g_1$ implies the same hypothesis for $n \mapsto n^{-it} g_1(n)$).  This completes the proof of part (i) of Corollary \ref{elliott-2}.

Part (ii) of Corollary \ref{elliott-2} is then deduced from Corollary \ref{elliott-2}(i) using precisely the same diagonalisation argument that was used to deduce Corollary \ref{elliott-1}(ii) from Corollary \ref{elliott-1}(i).
\end{proof}

\begin{remark} The above argument shows more generally that if the logarithmically averaged Elliott conjecture\footnote{One needs the variant where we sum over $X/\omega(X) \leq n \leq X$ rather than $n \leq X$.} (resp. Chowla conjecture) is proven for a given value of $k$, then the unweighted form of the Elliott conjecture (resp. Chowla conjecture) for that value of $k$ holds at almost all scales.  (Note in the case of the Chowla conjecture that the parameter $t$ will vanish, since $\lambda^k=1$ for even $k$ and $\lambda^k=\lambda$ does not pretend to be any twisted Dirichlet character for odd $k$.)
\end{remark}

\begin{remark}\label{rem1}
With small modifications, we can adapt the above proofs to prove Corollary \ref{cor1}. Firstly, by approximating the indicator function $1_{P^{+}(n)<P^{+}(n+1)}$ as in \cite[Section 4]{tera-binary} by a linear combination of indicator functions of the form $1_{P^{+}(n)<n^{\alpha}, P^{+}(n+1)<n^{\beta}}$, we can reduce the proof to showing
\begin{align}\label{eq9}
\lim_{X\to \infty; X\not \in \mathcal{X}_0}\mathbb{E}_{n\leq X}1_{P^{+}(n)<n^{\alpha}}1_{P^{+}(n+1)<n^{\beta}} =\rho(1/\alpha)\rho(1/\beta),   \end{align}
where $\rho(\cdot)$ is the Dickmann function and $\alpha, \beta \in (0,1)$ are any rational numbers. Since the set of rationals is countable, by a diagonal argument (as in the proof of Corollary \ref{elliott-2}(ii)) it suffices to prove \eqref{eq9} with $\alpha, \beta$ fixed. One starts by proving a version of the structural theorem (Theorem \ref{main}) in the case of the functions $g_1(n)=1_{P^{+}(n)<n^{\alpha}}$, $g_2(n)=1_{P^{+}(n)<n^{\beta}}$; these are not quite multiplicative functions, but they can be approximated as $1_{P^{+}(n)<n^{\alpha}}=1_{P^{+}(n)<X^{\alpha}}+O(1_{P^{+}(n)\in [(X/\log X)^{\alpha},X^{\alpha}]})$ for $n\in [X/\log X,X]$. The $O(\cdot)$ term has negligible contribution in the entropy decrement argument by standard estimates on smooth numbers, so the proof of Proposition \ref{dollop} goes through for the generalised limits associated to the correlations of $g_1$ and $g_2$ with $G\equiv 1$ (so certainly \eqref{gp-inv} holds). We did not use the specific properties of $g_1,g_2$ anywhere else in the proof of Theorem \ref{main}, so that proof goes through, giving 
\begin{equation}\label{eq28}\mathbb{E}_{d\in \mathbb{N}}|\lim^{*}_{X\to \infty}\mathbb{E}_{n\leq X/d}g_1(n)g_2(n+1)-c^{*}|=0\end{equation}
for all generalised limits $\lim^{*}$ and some constant $c^{*}$ depending on $\lim^{*}$. From \cite[Proof of Corollary 1.19]{tera-binary}, we have a logarithmic version of \eqref{eq9}, so following the proof of Corollary \ref{elliott-2} verbatim, we see that $c^{*}=\rho(1/\alpha)\rho(1/\beta)$. Then from Remark \ref{18-rem} we deduce \eqref{eq9}. We leave the details to the interested reader.
\end{remark}

\begin{proof}[Proof of Corollary \ref{chow}]
We observe from Corollary \ref{elliott-1}(i) (for odd $k$) or Corollary \ref{elliott-2}(ii) (for $k=2$) that for any distinct integers $h_1,\dots,h_k$ and $\eps>0$, one has
$$ |\E_{n \leq X} \lambda(n+h_1) \cdots \lambda(n+h_k)| \leq \eps $$
for all $X$ outside of a set $\mathcal{X}_{k,\varepsilon}$ of logarithmic Banach density zero, and hence also of logarithmic density zero.  The claim then follows by the same diagonalisation argument used to prove Corollary \ref{elliott-1}(ii) and Corollary \ref{elliott-2}(ii).
\end{proof}

\section{Consequences of the isotopy formulae}\label{isotopy-sec}

Before proving the isotopy formula in the form of Theorem \ref{isotopy}, let us state a variant of it that involves the quantities $f_d(a)$ present in Theorem \ref{main}. In what follows, a sequence $b_n$ of integers is said to be \emph{asymptotically rough} if for any given prime $p$, one has $p \nmid b_n$ for all sufficiently large $n$.  For instance, any increasing sequence of primes is asymptotically rough, as is the sequence $-1,-1,-1,\dots$.

\begin{lemma}\label{le1}  Let the notation and hypotheses be as in Theorem \ref{main}.  Let $n \mapsto \chi(n)n^{it}$ be a twisted Dirichlet character that weakly pretends to be $g_1 \cdots g_k$, if one exists; otherwise, choose $\chi$ and $t$ arbitrarily.  Let $a$ be an arbitrary integer.
\begin{itemize}
\item[(i)]  (Archimedean isotopy) For any natural number $h$, one has
$$ \lim_{X\to \infty}\E^{\log\log}_{d \leq X} |f_{hd}(a) - h^{-it} f_d(a)| = 0.$$
\item[(ii)]  (Non-archimedean isotopy) For any asymptotically rough sequence $b_n$ of natural numbers, one has
$$ \lim_{n\to \infty}\lim_{X\to \infty}\E^{\log\log}_{d \leq X} |f_{d}(ab_n) - \chi(b_n) f_d(a)|=0.$$
 In particular, since the sequence $b_n = -1$ is asymptotically rough, one has
\begin{equation}\label{lln}
\lim_{X\to \infty}\E^{\log\log}_{d \leq X} |f_{d}(-a) - \chi(-1) f_d(a)| = 0.
\end{equation}
\end{itemize}
\end{lemma}

A variant of  Lemma \ref{le1}(ii) (for logarithmic averaging, and with $b_n$ specialised to the primes in an arithmetic progression $1\ (q)$ for $q$ a period of $\chi$) was obtained in \cite[Corollary 3.7]{fh-furstenberg}. 
	
\begin{proof}  We may assume without loss of generality that $g_1 \cdots g_k$ weakly pretends to be $n \mapsto \chi(n)n^{it}$, as the claims follow from Theorem \ref{main}(i) otherwise.  Extracting out the contribution to \eqref{hd} from multiples of $h$, we see that
$$ \E^{\log\log}_{d \in \N} |f_{hd}(a) - f(a) (hd)^{-it}| = 0,$$
and also by \eqref{hd} we have
\begin{align*}
\E^{\log\log}_{d \in \N} |f_{d}(a) - f(a) d^{-it}| = 0.    
\end{align*}
Now the claim follows from the triangle inequality.  

To prove\footnote{Note that Lemma \ref{isotop} does not directly imply Claim (ii), since the Dirichlet character present in that lemma depends on the error $\varepsilon$.} Claim (ii), we observe from \eqref{hd} that 
\begin{align*}
\E^{\log\log}_{d \in \N} |f_{d}(ab_n) - f(ab_n)d^{-it}| =0
\end{align*}
for all $n$, and 
\begin{align*}
\E^{\log\log}_{d \in \N} |f_{d}(a) - f(a)d^{-it}| =0.
\end{align*}
Putting together the above two equalities we have
\begin{align}\label{eq11} 
\E^{\log\log}_{d \in \N} |f_{d}(ab_n) - \chi(b_n)f_d(a)| =|f(ab_n)-\chi(b_n)f(a)|
\end{align}
By Theorem \ref{main}, $f$ is the uniform limit of $\chi$-isotypic periodic functions $F_i$.  For each such $F_i$, we have $F_i(ab_n) = \chi(b_n) F_i(a)$ for all sufficiently large $n$, since the sequence $b_n$ is asymptotically rough. Thus also $f(ab_n)=\chi(b_n)f(a)+o_{n\to \infty}(1).$ Combining this with \eqref{eq11}, the claim follows. 
\end{proof}

We then use Lemma \ref{le1} to deduce the isotopy formulae (Theorem \ref{isotopy}).

\begin{proof}[Proof of Theorem \ref{isotopy}] We start with the proof of (i). By a diagonalisation argument, similarly as in the proof of Corollary \ref{elliott-1}(ii), it suffices to show that for any fixed rational $q>0$ there exists a set $\mathcal{X}_{0,q}$ of logarithmic density $0$ such that the claim holds with $\mathcal{X}_{0,q}$ in place of $\mathcal{X}_0$. Next, we argue that it suffices to consider the case $q\in \mathbb{N}$. Suppose that the case $q\in \mathbb{N}$ has been established, and let $q=a/b$ with $a,b\in \mathbb{N}$. Then if $\mathcal{X}_{0,q}:=(1/b)\mathcal{X}_{0,a}\cup (1/b)\mathcal{X}_{0,b}$ (which is still a set of logarithmic density zero), we have
\begin{align*}
&\lim_{X\to \infty; X\not \in \mathcal{X}_{0,q}}\left(\mathbb{E}_{n\leq X}g_1(n+h_1)\cdots g_k(n+h_k)-(a/b)^{it}\mathbb{E}_{n\leq bX/a}g_1(n+h_1)\cdots g_k(n+h_k)\right)\\
&=\lim_{X\to \infty; X\not \in \mathcal{X}_{0,q}}\left(b^{-it}\mathbb{E}_{n\leq bX}g_1(n+h_1)\cdots g_k(n+h_k)-(a/b)^{it}\mathbb{E}_{n\leq bX/a}g_1(n+h_1)\cdots g_k(n+h_k)\right)=0.
\end{align*}
Hence we may assume from now on that $q\in \mathbb{N}$. 
Observe that the statement of Lemma \ref{le1}(i) with $a=1$ can be written in the form
\begin{align*}
\mathbb{E}_{d\in \mathbb{N}}^{\log \log}|\lim^{*}_{X\to \infty}\mathbb{E}_{n\leq X/d}\big(g_1(n+h_1)\cdots g_k(n+h_k)-q^{-it}\mathbb{E}_{b\in \mathbb{Z}/q\mathbb{Z}}g_1(qn+b+h_1)\cdots g_k(qn+b+h_k)\big)|=0    
\end{align*}
for every generalised limit $\lim^{*}_{X\to \infty}$. By following the proof of Corollary \ref{elliott-1}(i) verbatim (see also Remark \ref{18-rem}), this leads to 
\begin{align}\label{eq12}
\lim_{X\to \infty;X\not \in \mathcal{X}_{0,q}} \mathbb{E}_{n\leq X}\big(g_1(n+h_1)\cdots g_k(n+h_k)-q^{-it}\mathbb{E}_{b\in \mathbb{Z}/q\mathbb{Z}}g_1(qn+b+h_1)\cdots g_k(qn+b+h_k)\big)=0    
\end{align}
for some set $\mathcal{X}_{0,q}$ of logarithmic density zero. But rewriting \eqref{eq12}, it becomes the identity asserted in Theorem \ref{isotopy}(i).

We turn to the proof of part (ii), which is similar. Again by a diagonalisation argument, it suffices to prove the statement for fixed $a$ rather than all $a$. From Lemma \ref{le1}(ii) we have
\begin{align*}
\mathbb{E}_{d\in \mathbb{N}}^{\log \log}|\lim^{*}_{X\to \infty}\mathbb{E}_{n\leq X/d}(g_1(n-ah_1)\cdots g_k(n-ah_k)-\chi(-1)g_1(n+ah_1)\cdots g_k(n+ah_k))|=0    
\end{align*}
for every generalised limit $\lim^{*}_{X\to \infty}$. Just as in the proof of part (i) of the Theorem, by the proof of Corollary \ref{elliott-1}(i) (cf. Remark \ref{18-rem}) we get
\begin{align*}
\lim_{X\to \infty;X\not \in \mathcal{X}_{0,a}} \mathbb{E}_{n\leq X}(g_1(n-ah_1)\cdots g_k(n-ah_k)-\chi(-1)g_1(n+ah_1)\cdots g_k(n+ah_k))|=0       
\end{align*}
for some set $\mathcal{X}_{0,a}$ of logarithmic density zero, and this is what we wished to prove.
\end{proof}

Morally speaking, the Archimedean isotopy formula implies that the argument of the correlation \eqref{eq19} becomes equidistributed at large scales whenever $t\neq 0$.  Unfortunately we cannot quite establish this claim as stated, because of the discontinuous nature of the complex argument function.  However, if we insert a continuous mollifier to remove this discontinuity, we can obtain equidistribution.  More precisely, we have the following result.

\begin{theorem}[Equidistribution of argument away from zero] \label{limits} Let $k\geq 1$, let $h_1,\ldots, h_k$ be integers and $g_1,\ldots, g_k:\mathbb{N}\to \mathbb{D}$ be $1$-bounded multiplicative functions. Suppose that the product $g_1\cdots g_k$ weakly pretends to be a twisted Dirichlet character $n\mapsto \chi(n)n^{it}$, where $t\neq 0$. Let us denote
\begin{align*}
S(X):=\mathbb{E}_{n\leq X}g_1(n+h_1)\cdots g_k(n+h_k).    
\end{align*}
Let $\psi: \C \to \C$ be a continuous function that vanishes in a neighbourhood of the origin, and let
$$ \overline{\psi}(z) \coloneqq \frac{1}{2\pi} \int_0^{2\pi} \psi(e^{i\theta} z)\ d\theta$$
be $\psi$ averaged over rotations around the origin.  Then we have
$$ \E^{\log}_{X \in \N} \psi(S(X)) - \overline{\psi}(S(X)) = 0.$$
\end{theorem}

\begin{proof}  Since $S$ is bounded, we may assume that $\psi$ is compactly supported.  By replacing $\psi$ by $\psi - \overline{\psi}$ we may assume that $\overline{\psi}=0$.  Approximating $\psi$ uniformly by partial Fourier series (e.g., using Fej\'er summation) in the angular variable, and using linearity, we may assume that $\psi$ takes the form $\psi(re^{i\theta}) = \Psi(r) e^{ik\theta}$ for some non-zero integer $k$ and some continuous compactly supported function $\Psi$ that vanishes in a neighbourhood of the origin (cf. the standard proof of the Weyl equidistribution criterion \cite{weyl}).  In particular we have the isotopy formula
\begin{equation}\label{psi-iso}
 \psi(\omega z) = \omega^k \psi(z)
\end{equation}
for all $z \in \C$ and $\omega \in S^1$.

Let $q > 1$ be an integer to be chosen later.
From Theorem \ref{isotopy}(i), outside of an exceptional set $\mathcal{X}_0$ of logarithmic density zero, we have
$$ \lim_{X \to \infty; X \not \in {\mathcal X}_0} S(X) - q^{it} S(X/q) = 0.$$
From \eqref{psi-iso} and the uniform continuity of $\psi$, this implies that
$$ \lim_{X \to \infty; X \not \in {\mathcal X}_0} \psi(S(X)) - q^{ikt} \psi(S(X/q)) = 0.$$
Taking logarithmic averages, we conclude that
$$ \E^{\log}_{X \in \N} \psi(S(X)) - q^{ikt} \psi(S(X/q)) = 0.$$
On the other hand, in analogy to \eqref{fd1d2}, we have the log-Lipschitz bound
\begin{align}\label{eq13}
|S(x)-S(y)|\ll |\log x-\log y|.   
\end{align}
We can use this and the uniform continuity of $\psi$ to estimate, for $X_0$ large enough,  
\begin{align*}
\mathbb{E}_{X\leq X_0}^{\log}\psi(S(X/q))&=\mathbb{E}_{X\leq X_0/q}^{\log}\mathbb{E}_{0\leq b<q} \psi(S(X+b/q))+o(1)\\
&=\mathbb{E}_{X\leq X_0/q}^{\log} \psi(S(X))+o(1)\\
&=\mathbb{E}_{X\leq X_0}^{\log} \psi(S(X))+o(1).
\end{align*}
Hence
$$ \E^{\log}_{X \in \N} \psi(S(X)) - \psi(S(X/q)) = 0.$$ By the triangle inequality, we conclude that
$$ (1 - q^{ikt}) \E^{\log}_{X \in \N} \psi(S(X)) = 0.$$
Since $t \neq 0$, we can select $q$ so that $q^{ikt} \neq 1$ for all $k\in \mathbb{N}$.  The claim follows.

\end{proof}

Suppose that $\Psi: [0,+\infty) \to [0,+\infty)$ is a non-negative continuous function vanishing near the origin, and let $I \subset \R/2\pi \Z$ be an arc in the unit circle $\R/2\pi \Z$.  Applying Theorem \ref{limits} to upper and lower approximants to the discontinuous function $z \mapsto \Psi(|z|) 1_I(\mathrm{arg}(z))$, and taking limits, we conclude that
$$ \E^{\log}_{X \in \N} (1_I(\mathrm{arg}(S(X))) - \frac{|I|}{2\pi}) \Psi(|S(X)|) = 0$$
where $|I|$ denotes the length of $I$.  Informally, this asserts that the argument $\mathrm{arg}(S(X))$ is uniformly distributed in the unit circle, so long as one inserts a continuous weight of the form $\Psi(|S(X)|)$.  It would be more aesthetically pleasing if we could replace this weight with a discontinuous cutoff such as $1_{|S(X)| \geq \eps}$, but we were unable to exclude the possibility that $|S(X)|$ lingers very close to $\eps$ for very many scales $X$, with the event that $|S(X)| \geq \eps$ being coupled in some arbitrary fashion to the argument of $S(X)$, leading to essentially no control on the argument of $S(X)$ restricted to the event $|S(X)| \geq \eps$.  On the other hand, if one was able to show that $S(X)$ did not concentrate at the origin in the sense that
$$ \limsup_{X_0 \to \infty} \E^{\log}_{X \leq X_0} 1_{|S(X)| \leq \eps} \to 0$$
as $\eps \to 0$, then the above arguments do show that 
$$ \E^{\log}_{X \in \N} 1_I(\mathrm{arg}(S(X))) = \frac{|I|}{2\pi}$$
for all intervals $I$, so that $\mathrm{arg}(S(X))$ is indeed asymptotically equidistributed on the unit circle.  Alternatively, by selecting the cutoff $\eps$ using the pigeonhole principle to ensure that $|S(X)|$ does not linger too often in a neighbourhood of $\eps$, one can prove statements such as the following: If $\delta > 0$, then for all sufficiently large $X_0$ outside of a set of logarithmic density zero, one can find $0 < \eps \leq \delta$ with the approximate equidistribution property 
$$ \E^{\log}_{X \leq X_0} \left(1_I(\mathrm{arg}(S(X))) - \frac{|I|}{2\pi}\right) 1_{|S(X)| \geq \eps} \leq \delta$$
for all intervals $I \subset \R/2\pi \Z$.
We leave the proof of this assertion to the interested reader.

Now we investigate the consequences of the non-archimedean isotopy formula (Theorem \ref{isotopy}(ii)). Many of these consequences tell us that the correlation \eqref{eq19} tends to $0$ along almost all scales also in some cases that are not covered by Corollary \ref{elliott-1}(i).

\begin{definition} We say that a tuple $(g_1,\ldots, g_k)$ of functions is \emph{reflection symmetric} if $g_i=g_{k+1-i}$ for all $1\leq i\leq \frac{k+1}{2}$. Similarly, we say that a tuple $(h_1,\ldots, h_k)$ of integers is \emph{progression-like} if $h_1+h_{k}=h_i+h_{k+1-i}$ for all $1\leq i\leq \frac{k+1}{2}$. In particular, all arithmetic progressions are progression-like.
\end{definition}

\begin{theorem} \label{even} Let $k\geq 1$ and let $h_1,\ldots, h_k$ be integers. Suppose that $\chi$ is an odd Dirichlet character (i.e., $\chi(-1)=-1$) with $\chi(n+h_1+h_k)=\chi(n)$ for all $n$. Let $g_1,\ldots, g_k:\mathbb{N}\to \mathbb{D}$ be multiplicative functions such that the product $g_1\cdots g_k$ weakly pretends to be a Dirichlet character $\psi$ with $\psi$ even. Suppose additionally that the tuple $(g_1,\ldots, g_k)$ is reflection symmetric and that the tuple $(h_1,\ldots, h_k)$ is progression-like. Then there exists an exceptional set $\mathcal{X}_0$ of logarithmic density $0$, such that
\begin{align*}
\lim_{X\to \infty; X\not \in \mathcal{X}_0}\mathbb{E}_{n\leq X} \chi(n)g_1(n+h_1)g_2(n+h_2)\cdots g_k(n+h_k)=0. 
\end{align*}

\end{theorem}

\begin{proof} Note that the function $g_1\cdots g_k\chi$ weakly pretends to be $\psi \chi$, which is an odd character. Hence by Theorem \ref{isotopy}(ii) there exists some set $\mathcal{X}_0$ of logarithmic density $0$, such that for $X\not \in \mathcal{X}_0$ we have
\begin{align*}
&\mathbb{E}_{n\leq X} \chi(n)g(n)g_1(n+h_1)\cdots g_k(n+h_k)\\
&=-\mathbb{E}_{n\leq X} \chi(n)g_1(n-h_1)g(n-h_2)\cdots g(n-h_k)+o(1).
\end{align*}
By translation invariance, the periodicity assumption $\chi(n+h_1+h_k)=\chi(n)$, and the progression-likeness of $(h_1,\ldots, h_k)$, the latter expression equals
\begin{align*}
&-\mathbb{E}_{n\leq X} \chi(n+h_1+h_k)g_1(n+h_k)g_2(n+h_1+h_k-h_2)\cdots g_k(n+h_1)+o(1)\\
&=-\mathbb{E}_{n\leq X}\chi(n)g_1(n+h_k)g_2(n+h_{k-1})\cdots g_k(n+h_1)+o(1).
\end{align*}
Since the tuple $(g_1,\ldots, g_k)$ is reflection symmetric, this equals the the original correlation with a minus sign, proving the statement. 
\end{proof}

Corollary \ref{even-chowla} is an immediate consequence of Theorem \ref{even}.

\begin{proof}[Proof of Corollary \ref{even-chowla}]
Taking $g_1=\ldots=g_k=\lambda$ and $(h_1,\ldots, h_k)=(0,a,\ldots, (k-1)a)$ in Theorem \ref{even}, we readily obtain the claim.
\end{proof}

In other words, the shifted products of the Liouville function can be shown to be orthogonal to some suitable Dirichlet characters also when there is an even number of shifts. As already mentioned, also the weaker, logarithmic version of Corollary \ref{even-chowla} is new.

The next theorem is in the same spirit as Theorem \ref{even}, but with somewhat different conditions.

\begin{theorem} \label{odd} Let $k\geq 1$ be odd, and let  $g_1,\ldots, g_k:\mathbb{N}\to \mathbb{D}$ be multiplicative functions such that the product $g_1\cdots g_k$ weakly pretends to be a Dirichlet character $\chi$ with $\chi$ odd. Suppose also that the tuple $(g_1,\ldots, g_k)$ is reflection symmetric and that $(h_1,\ldots, h_k)$ is a progression-like tuple of integers. Then there exists an exceptional set $\mathcal{X}_0$ of logarithmic density $0$, such that
\begin{align*}
\lim_{X\to \infty; X\not \in \mathcal{X}_0}\mathbb{E}_{n\leq X} g_1(n+h_1)g_2(n+h_2)\cdots g_k(n+h_k)=0.   
\end{align*}
\end{theorem}

\begin{proof} As with Theorem \ref{even}, this follows directly from the isotopy formula (Theorem \ref{isotopy}) and translation invariance.  
\end{proof}

This theorem can for example be applied to the variants
\begin{align*}
 \lambda_q(n):=e\left(\frac{2\pi i\Omega(n)}{q}\right)  \end{align*}
of the Liouville function that take values in the $q$th roots of unity. Here $\Omega(n)$ is the number of prime factors of $n$ with multiplicities. We obtain the following.

\begin{corollary}\label{cor-1}  Let $k\geq 1$ be odd, $q\in \mathbb{N}$, and let $\chi$ be an odd Dirichlet character. Then there exists an exceptional set $\mathcal{X}_0$ of logarithmic density $0$, such that
\begin{align*}
\lim_{X\to \infty; X\not \in \mathcal{X}_0}\mathbb{E}_{n\leq X} \lambda_q(n)\chi(n)\lambda_q(n+a)\chi(n+a)\cdots \lambda_q(n+(k-1)a)\chi(n+(k-1)a)=0.   
\end{align*}
\end{corollary}

\begin{proof} We apply Theorem \ref{odd} with $g_j(n)=g(n):=\chi(n)\lambda_q(n)$ and $(h_1,\ldots, h_k)=(0,\ldots, (k-1)a)$. Then if $q\nmid k$, the function $g^k\chi^k$ does not weakly pretend to be any twisted Dirichlet character, since $g^k$ does not do so. In this case, we may appeal to  Corollary \ref{elliott-1}(i) to obtain the claim. Suppose then that $q\mid k$. Then $g^k$ weakly pretends to be $\chi^k$, which is an odd character, so Theorem \ref{odd} is applicable.  
\end{proof}

\begin{example} Let $\chi_3$ be the odd Dirichlet character of modulus $3$ and $\chi_8$ any odd Dirichlet character of modulus $8$. Then from Corollary \ref{cor-1} and partial summation, for any sequences $1\leq \omega_m\leq x_m$ of reals tending to infinity we have
\begin{align*}
\lim_{m\to \infty}\mathbb{E}^{\log}_{x_m/\omega_m\leq n\leq x_m}\chi_3(n)\lambda_3(n)\lambda_3(n+3)\lambda_3(n+6)=0    
\end{align*}
and 
\begin{align*}
\lim_{m\to \infty}\mathbb{E}^{\log}_{x_m/\omega_m\leq n\leq x_m}\lambda_3(n)\lambda_3(n+2)\lambda_3(n+4)\chi_8(n+6)=0.    
\end{align*}

This is seen by applying the corollary to the functions $g_j(n)=\lambda_3(n)\chi_3(n)$ with $a=3$ and $g_j(n)=\lambda_3(n)\chi_8(n)$ with $a=2$ and using $n(n+2)(n+4)\equiv n+6 \pmod{8}$ for $n$ odd.
\end{example}

We then turn to bounding more general correlations of multiplicative functions where the shifts involved no longer form a progression-like tuple. In the case of triple correlations, we obtain savings that are explicit but nevertheless far from the desired $o(1)$ bound. 

\begin{theorem}[Savings in logarithmic three-point Elliott conjecture] Let $g:\mathbb{N}\to \mathbb{D}$ be a multiplicative function, and let $h_1,h_2,h_3$ be distinct integers. Suppose that $g$ is non-pretentious in the sense that
\begin{align*}
\liminf_{X\to \infty}\inf_{|t|\leq X}\mathbb{D}(g,n\mapsto\chi(n)n^{it},x)=\infty    
\end{align*}
for every Dirichlet character $\chi$. Then for any sequences $1\leq \omega_m\leq x_m$ tending to infinity we have
\begin{align}\label{eqtn6}
\limsup_{m\to \infty}|\mathbb{E}^{\log}_{x_m/\omega_m\leq n\leq x_m} g(n+h_1)g(n+h_2)g(n+h_3)|\leq \frac{1}{\sqrt{2}}.
\end{align}

\end{theorem}

\begin{remark} This looks superficially similar to \cite[Lemma 5.3]{klurman-mangerel} (and also to \cite[Proposition 7.1]{tt}, which achieves the better upper bound of $\frac{1}{2}$ for real-valued multiplicative functions). However, importantly, the shifts $h_i$ are allowed to be arbitrary here, while in the aforementioned results they had to form an arithmetic progression for the method to work.
\end{remark}

\begin{proof}  If $h_1,h_2,h_3$ is an arithmetic progression, we may apply \cite[Lemma 5.3]{klurman-mangerel}, so we may henceforth suppose that $h_1,h_2,h_3$ is not an arithmetic progression.

If the function $g^3$ does not weakly pretend to be any Dirichlet character, we get the bound $0$ for the $\limsup$ by \cite[Theorem 1.2(ii)]{tt}. Suppose then that $g^3$ weakly pretends to be some character $\chi$. By the isotopy formula (Theorem \ref{isotopy}), partial summation and translation invariance, we have 
\begin{align}\label{eqtn5}
\begin{aligned}
&\mathbb{E}^{\log}_{x_m/\omega_m\leq n\leq x_m}g(n+h_1)g(n+h_2)g(n+h_3)\\
&=\chi(-1)\mathbb{E}^{\log}_{x_m/\omega_m\leq n\leq x_m}g(n+h_1)g(n+h_1+h_3-h_2)g(n+h_3)+o_{m\to \infty}(1).
\end{aligned}
\end{align}
In particular, the first part of \eqref{eqtn5} is the average of both parts of the equation. Hence, the average on the left-hand side of \eqref{eqtn6} is up to $o_{m\to \infty}(1)$ bounded by
\begin{align*}
&\frac{1}{2}|\mathbb{E}^{\log}_{x_m/\omega_m\leq n\leq x_m}g(n+h_1)g(n+h_3)(g(n+h_2)+\chi(-1)g(n+h_1+h_3-h_2))|\\
&\leq \frac{1}{2}\mathbb{E}^{\log}_{x_m/\omega_m\leq n\leq x_m} |g(n+h_2)+\chi(-1)g(n+h_1+h_3-h_2))|.   
\end{align*}
By the Cauchy-Schwarz inequality, this is bounded by
\begin{align*}
&\frac{1}{2}(\mathbb{E}^{\log}_{x_m/\omega_m\leq n\leq x_m} |g(n+h_2)+\chi(-1)g(n+h_1+h_3-h_2))|^2)^{\frac{1}{2}}\\
&\leq\frac{1}{2}(\mathbb{E}^{\log}_{x_m/\omega_m\leq n\leq x_m}(2+2\chi(-1)\textnormal{Re}(g(n+h_2)\overline{g}(n+h_1+h_3-h_2))))^{\frac{1}{2}}.    
\end{align*}
Since $h_2\neq h_1+h_3-h_2$ by assumption, we can apply \cite[Corollary 1.5]{tao} to see that the term involving real parts contributes $o_{m\to \infty}(1)$. Then we indeed get a bound of $\frac{1}{\sqrt{2}}+o_{m\to \infty}(1)$ for the correlation.
\end{proof}

\begin{remark} For specific multiplicative functions one can do slightly better by not applying Cauchy-Schwarz. For example, in the case $g(n)=\lambda_3(n)$ one gets a bound of $\frac{2}{3}$ for the correlation by using the fact (following from \cite[Corollary 1.5]{tao}) that $(\lambda_3(n),\lambda_3(n+h))$ takes for fixed $h\neq 0$ each of the possible $9$ values with logarithmic density $1/9$.
\end{remark}

\section{The case of few sign patterns}\label{sign-sec}

In this section we prove Theorem \ref{sign-thm}.  Assume the hypotheses of that theorem.  Let $h$ be a natural number.  By the Hahn-Banach theorem, it suffices to show that
$$ \E^*_{n \in \N} \lambda(n) \lambda(n+h) = 0$$
for every generalised limit $\lim^*_{X \to \infty}$.  Accordingly, let us fix such a limit.  As usual, we introduce the correlation sequences
\begin{equation}\label{fda-def-again}
 f_d(a) \coloneqq \lim^*_{X \to \infty} \E_{n \leq X/d} \lambda(n) \lambda(n+ah)
\end{equation}
for every real $d>0$.  Our task is now to show that
$$ f_1(1) = 0.$$
Proposition \ref{dollop} (noting that $G(p) = 1$ in our case) establishes the approximate isotopy formula
$$ \sup_{d >0} \E^{\log}_{p \leq P} |f_{dp}(a) - f_d(ap)| \leq \eps$$
whenever $\eps>0$ and $P$ is sufficiently large depending on $\eps$.  But because of our hypothesis of few sign patterns, we can obtain a stronger result in which the logarithmic weighting on the averages is removed.

\begin{proposition}[Improved approximate isotopy formula]  Let $f_d(a)$ be as in \eqref{fda-def-again}, let $\eps>0$, and let $a$ be a natural number.  Assume the hypotheses of Theorem \ref{sign-thm}.  Then there exist arbitrarily large $m$ such that
$$ \sup_{d > 0} \E_{2^m \leq p < 2^{m+1}} |f_{dp}(a) - f_d(ap)| \leq \eps.$$
\end{proposition}

This formula also applies for negative $a$, but in this argument we only require the case of positive $a$ (in fact, for the binary correlations considered here, we only need the case $a=1$).

\begin{proof}  This will be a modification of the arguments in \cite[\S 3]{tt}, and we freely use the notation from that paper.  

Let $d>0$ be real, let $a$ be an integer, and let $m$ be a large integer to be chosen later.  We allow implied constants to depend on $h,a$, but they will remain uniform in $d,m,\eps$.  From \eqref{fdag} we have the formula
$$
f_{dp}(a) - f_d(ap) = \E^{(d)} \mathbf{g}^{(d)}(0) \mathbf{g}^{(d)}(aph) (p 1_{p|\mathbf{n}^{(d)}}-1) + O(\eps)
$$
for all $2^m \leq p < 2^{m+1}$, if $m$ is sufficiently large depending on $\eps$, and where $\mathbf{g}^{(d)} = \mathbf{g}^{(d)}_1 =\mathbf{g}^{(d)}_2$ and $\mathbf{n}^{(d)}$ are the random variables provided by Proposition \ref{corr} (with $g_1=g_2=\lambda$).  We can thus write the expression
$$ \E_{2^m \leq p < 2^{m+1}} |f_{dp}(a) - f_d(ap)| $$
as
$$
\E^{(d)} \E_{2^m \leq p < 2^{m+1}} c_p \mathbf{g}^{(d)}(0) \mathbf{g}^{(d)}(aph) (p 1_{p|\mathbf{n}^{(d)}}-1) + O(\eps)$$
for some sequence of complex numbers $c_p$ with $|c_p| \leq 1$.  By stationarity we can also write this expression as
$$
\E^{(d)} \E_{1 \leq l \leq 2^m} \E_{2^m \leq p < 2^{m+1}} c_p \mathbf{g}^{(d)}(l) \mathbf{g}^{(d)}(l+aph) (p 1_{\mathbf{n}^{(d)} = -l\ (p)}-1) + O(\eps)$$
and thus
$$ \E_{2^m \leq p < 2^{m+1}} |f_{dp}(a) - f_d(ap)| = \E^{(d)} F( \mathbf{X}^{(d)}, \mathbf{Y}^{(d)} )$$
where $\mathbf{X}^{(d)} = \mathbf{X}^{(d)}_m \in \{-1,+1\}^{(2ah+1) 2^m}$, $\mathbf{Y}^{(d)} = \mathbf{Y}^{(d)}_m \in \prod_{2^m \leq p < 2^{m+1}} \Z/p\Z$ are the random variables
\begin{align*}
\mathbf{X}^{(d)} &\coloneqq \left( \mathbf{g}^{(d)}(l) \right)_{1 \leq l \leq (2ah+1) 2^m} \\
\mathbf{Y}^{(d)} &\coloneqq ( \mathbf{n}^{(d)}\ (p))_{2^m \leq p < 2^{m+1}}
\end{align*}
and $F: \{-1,+1\}^{(2ah+1) 2^m} \times \prod_{2^m \leq p < 2^{m+1}} \Z/p\Z \to \C$ is the function
\begin{align*}
& F( (g_l)_{1 \leq l \leq (2ah+1)2^m}, (n_p)_{2^m \leq p < 2^{m+1}})\\
&\quad \coloneqq \E_{1 \leq l \leq 2^m} \E_{2^m \leq p < 2^{m+1}} c_p g_l g_{l+aph} (p 1_{n_p=-l\ (p)} - 1).
\end{align*}
Repeating the arguments in \cite[\S 3]{tt} verbatim (but without the additional conditioning on the $\mathbf{Y}_{<m}$ random variable), we conclude that
$$ \E^{(d)} |F( \mathbf{X}^{(d)}, \mathbf{Y}^{(d)} )| \leq \eps $$
unless we have the mutual information bound
$$ \mathbf{I}( \mathbf{X}^{(d)} : \mathbf{Y}^{(d)} ) > \eps^5 \frac{2^m}{m}.$$
At this point we deviate from the arguments in \cite[\S 3]{tt} by using the trivial bound
$$ \mathbf{I}( \mathbf{X}^{(d)} : \mathbf{Y}^{(d)} ) \leq \mathbf{H}( \mathbf{X}^{(d)} )$$
to conclude that we will have the desired bound
$$ \E_{2^m \leq p < 2^{m+1}} |f_{dp}(a) - f_d(ap)| \leq \eps $$
whenever $\mathbf{X}^{(d)}$ obeys the entropy bound
$$ \mathbf{H}( \mathbf{X}^{(d)} ) \leq \eps^5 \frac{2^m}{m}.$$
By Jensen's inequality, this bound will hold if $\mathbf{X}^d$ attains at most $\exp( \eps^5 \frac{2^m}{m} )$ values with positive probability.
Using the correspondence principle (Proposition \ref{corr}), this claim in turn is equivalent to the number of possible sign patterns $(\lambda(n+l))_{1 \leq l \leq (2ah+1) 2^m}$ not exceeding $\exp( \eps^5 \frac{2^m}{m} )$; note that this assertion does not depend on $d$, so we in fact obtain the uniform bound
$$ \sup_{d>0} \E_{2^m \leq p < 2^{m+1}} |f_{dp}(a) - f_d(ap)| \leq \eps $$
in this case.  But by the hypothesis of Theorem \ref{sign-thm}, this assertion holds for arbitrarily large values of $m$.
\end{proof}

Now we establish Theorem \ref{sign-thm}.  By the above proposition, for any $\eps>0$, there exist arbitrarily large $m$ such that
$$ f_1(1) = \E_{2^m \leq p < 2^{m+1}} f_{P}(p) + O( \eps ),$$
where $P:=2^m$. By \eqref{fda-def-again}, it suffices to show that
$$ \limsup_{X \to \infty} |\E_{P \leq p<2P} \E_{n \leq X/P} \lambda(n) \lambda(n+ph)| \ll \eps$$
for sufficiently large $P$.  But this follows from the results in \cite[\S 3]{tao}, specifically Lemmas 3.6, 3.7 and equation (2.9) of that paper\footnote{When applying these results, note that (in the notation of \cite[Lemma 3.6]{tao}) the length of the sum over $j\in [1,H]$ and the range $\mathcal{P}_H$ of primes $p$ do not need to be controlled by the same parameter $H$.} (see also Remark 3.8 for a simplification in the case of the Liouville function).  We remark that the equation \cite[(2.8)]{tao} relies crucially on the Matom\"aki-Radziwi{\l}{\l} theorem \cite{mr} (as applied in \cite{mrt}).

\begin{remark}  A similar argument also gives the odd order cases of the Chowla conjecture if one strengthens the hypothesis of Theorem \ref{sign-thm} to hold for \emph{all} sufficiently large $K$, rather than for arbitrarily large $K$, by using the arguments in \cite[\S 3]{tt-odd} (but with the exceptional sets ${\mathcal M}_1$ in those arguments now being empty, and using unweighted averaging in $n$ rather than logarithmic averaging).  We leave the details to the interested reader.
\end{remark}

\end{document}